\PassOptionsToPackage{lowtilde}{url}		
\documentclass[11pt]{amsart}
\pdfoutput=1

\usepackage[english]{babel}
\usepackage{amssymb}
\usepackage{xfrac}
\usepackage{amsthm}
\usepackage{mathtools}                          
\usepackage{mathrsfs}				    
\usepackage{verbatim}				    
\usepackage[all]{xy}    	          
\SelectTips{cm}{11}                  
\usepackage[T1]{fontenc}        
\usepackage{graphicx}	
\usepackage[normalem]{ulem}          
\usepackage{microtype}               
\usepackage{lmodern}

\usepackage{enumitem}
\setenumerate[1]{label=(\roman*),ref=\thesubsection.{\roman*}}    

\usepackage{hyperref}
\hypersetup{
  colorlinks   = true,          
  urlcolor     = blue,          
  linkcolor    = blue,          
  citecolor   = red             
}
\newcommand{\A}{{\mathbb{A}}}
\newcommand{\N}{{\mathbb{N}}}
\newcommand{\Z}{{\mathbb{Z}}}
\newcommand{\Q}{{\mathbb{Q}}}

\newcommand{\F}{{\mathbb F}}

\newcommand{\calA}{{\mathcal A}}
\newcommand{\calB}{{\mathcal B}}
\newcommand{\calC}{{\mathcal C}}

\newcommand{\calM}{{\mathcal M}}

\newcommand{\calO}{{\mathcal O}}

\newcommand{\frakV}{{\mathfrak V}}

\newcommand{\frakM}{{\mathfrak M}}
\newcommand{\frakU}{{\mathfrak U}}

\newcommand*\quot[2]{{^{\textstyle #1}\big/_{\textstyle #2}}}%
\newcommand{\bigslant}[2]{{\raisebox{.2em}{$#1$}\left/\raisebox{-.2em}{$#2$}\right.}}

\newcommand{\X}{\mathscr X}			
\newcommand{\s}{\mathscr S}			
\newcommand{\T}{\mathscr Z}			

\newcommand{\OO}{\mathcal O}		%
\newcommand{\OOO}{\mathcal O^\circ}		%
\newcommand{\id}{\mathrm{id}}

\newcommand*{\longhookrightarrow}{\ensuremath{\lhook\joinrel\relbar\joinrel\rightarrow}}

\DeclareMathOperator{\Spf}{Spf}
\DeclareMathOperator{\Spec}{Spec}

\DeclareMathOperator{\Hom}{Hom}

\DeclareMathOperator{\Gal}{Gal}

\DeclareMathOperator{\NL}{NL}
\DeclareMathOperator{\ord}{ord}

\newtheorem{theorem}[subsection]{Theorem}
\newtheorem{corollary}[subsection]{Corollary}
\newtheorem{lemma}[subsection]{Lemma}
\newtheorem{proposition}[subsection]{Proposition}

\theoremstyle{definition}

\theoremstyle{remark}
\newtheorem{remark}[subsection]{Remark}
\newtheorem{example}[subsection]{Example}

\title{Galois descent of semi-affinoid spaces}

\date{\today}

\author{Lorenzo Fantini}
\address{Institut Mathématique de Jussieu, Université Pierre et Marie Curie, Paris, France}
\email{\href{mailto:lorenzo.fantini@imj-prg.fr}{lorenzo.fantini@imj-prg.fr}}
\urladdr{\url{https://webusers.imj-prg.fr/~lorenzo.fantini}}%

\author{Daniele Turchetti}
\address{Laboratoire de Mathématiques Nicolas Oresme, Université de Caen Basse-Normandie, Caen, France}
\email{\href{mailto:daniele.turchetti@unicaen.fr}{daniele.turchetti@unicaen.fr}}
\urladdr{\url{https://turchetti.users.lmno.cnrs.fr/}}

\usepackage{etoolbox}
\makeatletter
\patchcmd{\@maketitle}
  {\ifx\@empty\@dedicatory}
  {\ifx\@empty\@date \else {\vskip3ex \centering\footnotesize\@date\par\vskip1ex}\fi
   \ifx\@empty\@dedicatory}
  {}{}
\patchcmd{\@adminfootnotes}
  {\ifx\@empty\@date\else \@footnotetext{\@setdate}\fi}
  {}{}{}
\makeatother

\begin{document}

\begin{abstract}
In this paper we study the Galois descent of semi-affinoid non-archimedean analytic spaces.
These are the non-archimedean analytic spaces which admit an affine special formal scheme as model over a complete discrete valuation ring, such as for example open or closed polydiscs or polyannuli.
Using Weil restrictions and Galois fixed loci for semi-affinoid spaces and their formal models, we describe a formal model of a $K$-analytic space $X$, provided that $X\otimes_KL$ is semi-affinoid for some finite tamely ramified extension $L$ of $K$.
As an application, we study the forms of analytic annuli that are trivialized by a wide class of Galois extensions that includes totally tamely ramified extensions.
In order to do so, we first establish a Weierstrass preparation result for analytic functions on annuli, and use it to linearize finite order automorphisms of annuli.
Finally, we explain how from these results one can deduce a non-archimedean analytic proof of the existence of resolutions of singularities of surfaces in characteristic zero.
\end{abstract}

\maketitle

\setcounter{tocdepth}{1}

\section{Introduction}

Let $K$ be a field which is complete with respect to a non-archimedean absolute value, and let $R$ be its valuation ring.
In this paper we study the Galois descent of semi-affinoid $K$-analytic spaces, that are those non-archimedean $K$-analytic spaces that admit as a model a formal $R$-scheme that is affine and special (that is, formally topologically of finite type).
Roughly speaking, they correspond to the analytic spaces that are bounded (without necessarily being compact), such as open polydiscs and their closed subspaces.
The underlying idea is that semi-affinoid spaces have enough bounded analytic functions to be determined by them, and that they are simpler to study via their $R$-algebras of bounded functions, rather than via the whole $K$-algebra of analytic functions, which for example may not be noetherian.

Let $K'|K$ be a finite extension and let $R'|R$ be the corresponding extension of valuation rings.
We are interested in how one can determine a formal $R$-model of a $K$-analytic space $X$, knowing an affine formal $R'$-model $\X$ of the base change $X\otimes_KK'$ and the action of the Galois group of $K'|K$ on the latter.
We are able to describe the situation completely in the case when $K'$ is a finite tamely ramified Galois extension of $K$.
More precisely, we prove that $X$ is itself semi-affinoid, and has as  $R$-model the Galois-fixed locus of the (dilated) Weil restriction of $\X$ to $R$; this is the content of Theorem~\ref{theorem_descent_model}.

This is inspired by work of Edixhoven \cite{Edixhoven92}, who used this technique to study the behavior of Néron models of abelian varieties under totally tamely ramified base field extensions.
To carry out these constructions in our setting, we need to study the problem of the representability of the Weil restrictions of semi-affinoid spaces and their models, building on results of Bertapelle \cite{Bertapelle00}, and define the dilated Weil restriction, a variant of the Weil restriction of special formal $R$-schemes.
While being merely a computational tool, the dilated Weil restriction allows us to describe explicitly and in a simple way a formal model of the Weil restriction of a semi-affinoid space under a tame extension, which proves to be very useful in practice.
This can be thought of as a correction to the fact that, for ramified extensions, the Weil restriction to $R$ of a special formal $R'$-scheme $\X$ is not a model of the Weil restriction to $K$ of the $K'$-analytic space associated with $\X$. 

Semi-affinoid $K$-analytic spaces appear naturally as fibers of closed points of the specialization map that goes from a $K$-analytic space $X$ to the special fiber of an $R$-model (for example a semi-stable or, more generally, a strict normal crossing model) of $X$.
If $X$ is a smooth and proper $K$-analytic curve and $\X$ is a semi-stable $R$-model of $X$, then all those fibers are open discs and open annuli.
This relation between the structure of non-archimedean analytic curves and their semi-stable reduction goes back to work of Bosch and Lütkebohmert \cite{BosLut85}, where it was used to give a non-archimedean analytic proof of the semi-stable reduction theorem of Deligne and Mumford.
Since in general semi-stable models of a curve exist only after a finite separable base change $K'|K$, it is natural to study the $K$-forms of $K'$-analytic discs and annuli, that are those (strictly) $K$-analytic spaces which become isomorphic to a disc or to an annulus over $K'$.

Tamely ramified forms of discs are well understood.
Ducros \cite{Ducros13} proved that if $K'|K$ is tamely ramified and $V$ is a $K$-analytic space such that $V\otimes_KK'$ is an open (poly)disc, then $V$ is itself a (poly)disc, while the analogous result for (one dimensional) closed discs was proven by Schmidt \cite{Schmidt15}.
As an application of our descent machinery, we devote our attention to the study of forms of annuli, both open and closed.
Unlike the case of discs, annuli admit tamely ramified forms that are not themselves annuli.
In fact, for some $K$-forms $V$ that become annuli under a quadratic extension $K'|K$, the Galois group $\Gal(K'|K)$ may switch the branches of the annulus $V\otimes_KK'$, that means exchange the two irreducible components of its canonical reduction, but this can only happen if $V$ is not an annulus itself.
In Theorem~\ref{theorem_forms_switched_branches} we show that this is indeed possible, and classify these forms up to isomorphism.
On the other hand, we prove in Theorem~\ref{theorem_annuli} that if the extension $K'|K$ is Galois and reasonably well behaved (that is, its residue characteristic does not divide $[K':K]$, its residual extension is solvable, and $K$ contains the $[K':K^{\mathrm{ur}}]$-th roots of unity) and $\Gal(K'|K)$ does not switch the branches of $V\otimes_KK'$, then $V$ is an annulus.
Combining the two results, when $K'|K$ is a quadratic extension we obtain a complete description of the $K$-forms of a $K'$-annulus $X$; or in other words we determine the set $\mathrm{H}^1\big(\Gal(K'|K),\mathrm{Aut}_{K'}(X)\big)$ arising from group cohomology.

The results of Theorem~\ref{theorem_annuli} for all tamely ramified extensions of complete valued fields have been independently proven by Chapuis in \cite{Chapuis2017}, where he also treats the cases of polyannuli and closed polydiscs.
Note that our tools also allow to retrieve the results of Ducros (in dimension one) and Schmidt, provided that the extension turning the $K$-forms into discs is well-behaved in the sense discussed above.
We believe that this is interesting in itself, since our techniques are quite different from those of Ducros, Schmidt, and Chapuis, which are all based on Temkin's theory of graded reduction and on the graded counterparts of several classical algebra results.
Note that their use of graded reductions allows them to work over non-discretely valued fields and to also treat non-strict discs and annuli.
While our techniques do not apply in those cases, they are much more effective to study non-trivial forms, as Theorem~\ref{theorem_forms_switched_branches} shows.

Two $K$-analytic discs are isomorphic over $K$ if and only if their radii differ by an element of the value group $|K|$ of $K$.
The moduli space of $K$-annuli is a bit richer : we describe it in Theorem~\ref{theorem_moduli_space_fractional_annuli}.
As a consequence, we can completely classify up to isomorphisms the forms we obtain in Theorem~\ref{theorem_annuli}.

To prove the results on forms of Section~\ref{section_forms_annuli}, we compute explicitly the Galois-fixed locus of the Weil restriction of an affine $R$-formal model $\X$ of an annulus $V\otimes_KK'$.
An essential ingredient in these computations is a complete description of the possible Galois actions on annuli.
To achieve this, using techniques reminiscent of the theory of Newton polygons we establish a Weierstrass preparation result for functions on both open and closed annuli (Proposition~\ref{proposition_weierstrass_preparation}), and then deduce some linearization results for tame finite order automorphisms of the algebras of annuli (Propositions~\ref{proposition_linearization} and \ref{proposition_linearization_switched_branches}).
The idea of obtaining a linearization result via the Weierstrass preparation theorem is already present in work of Henrio \cite{Henrio1999,Henrio01}, but we are able to generalize his results by allowing closed annuli and a much more general class of automorphisms.
We are convinced that these results are of independent interest besides the applications in the present paper.

The reason we grew interested in the study of forms of annuli was in relation with the first author's previous work \cite{Fantini2017}.
There, he developed a theory of non-archimedean links that provides a solid bridge between the birational geometry of surfaces over $k$ and the theory of semi-stable reduction for curves over $k((t))$.
At the end of the paper we explain how our results on forms of annuli, combined with the techniques of \cite{Fantini2017}, yield a proof of the existence of resolution of singularities for surfaces over an algebraically closed field of characteristic zero (Theorem~\ref{theorem_resolutions}), a classical result of Zariski \cite{Zariski1939}.
This new proof, which is completely non-archimedean analytic in spirit, is inspired by the existing analytic proofs of the semi-stable reduction theorem, as for example in \cite{Ducros}.

Semi-affinoid spaces and their forms appear naturally also in arithmetic geometry.
Let $R$ be be the valuation ring of a $p$-adic field and let $\F_q$ be its residue field.
It is a central problem in the $p$-adic local Langlands program to study the modularity of the liftings of a Galois representation $\overline\rho\colon\Gal(\mathbb Q_p^\mathrm{alg}|\mathbb Q_p)\to\mathrm{GL}_2(\F_q)$ to a representation with values in $\mathrm{GL}_2(R)$.
The rings of deformations of such liftings can be studied with the help of a $\F_q$-variety, the Kisin variety of $\overline\rho$, and the deformations associated with a given closed point of the Kisin variety admit a canonical semi-affinoid structure.
Since this study is generally simpler after passing to a finite and totally tamely ramified Galois extension (see for example \cite{CarusoDavidMezard2016}), it would be interesting to describe the forms of such semi-affinoid spaces.

\medskip

Let us now give a short overview of the content of the paper.
In Section~\ref{section_semi-affinoids}, we introduce semi-affinoid spaces and describe their basic properties.
In Section~\ref{section_weil}, we discuss the Weil restriction functor and its representability for semi-affinoid spaces and their models.
The dilated Weil restriction of an affine special formal scheme is also studied there, while the $G$-fixed locus functors are treated in Section~\ref{section_fixed_locus}.
Section~\ref{section_descent_model} contains the main descent result, Theorem~\ref{theorem_descent_model}, which describes a formal model of a $K$-form of a semi-affinoid $K'$-analytic space $X$ in terms of the $G$-fixed locus of the dilated Weil restriction of a model of $X$.
We then move to the study of annuli, which are defined in Section~\ref{section_annuli}, where we also prove the Weiestrass preparation for functions on annuli (Proposition~\ref{proposition_weierstrass_preparation}), and use it to deduce linearization results for the tame automorphisms of annuli (Propositions~\ref{proposition_linearization} and \ref{proposition_linearization_switched_branches}).
In Section~\ref{section_fractional_annuli} we define annuli with fractional moduli, and describe their moduli space.
Finally, Section~\ref{section_forms_annuli} is devoted to the study of forms of annuli: the triviality in the case of forms with fixed branches is proven in Theorem~\ref{theorem_annuli}, while the case of switched branches is addressed in Theorem~\ref{theorem_forms_switched_branches}.

\subsection*{Notation}
Throughout the paper, we denote by $K$ a field which is complete with respect to a (non-trivial) discrete valuation, by $R$ its valuation ring, by $\pi$ a uniformizer of $R$, and by $k=R/\pi R$ the residue field of $K$.
We denote by $K'$ a finite extension of $K$, by $R'$ its valuation ring, by $\varpi$ a uniformizer of $R'$, and by $k'=R'/\varpi R'$ the residue field of $K'$.
\\
We use straight letters, such as $X$, to denote non-archimedean analytic spaces (over $K$ or $K'$), while curly letters, such as $\X$, will be reserved for formal schemes (over $R$ or $R'$).
\\
All the affinoid algebras and affinoid spaces that we consider are strict.
In particular, all our results can be interpreted in the framework of rigid geometry or in any of the other languages of non-archimedean analytic geometry.
That said, the authors are partial to Berkovich theory, and so the text contains a few remarks about the geometry of the Berkovich spaces underlying the analytic spaces considered.

\subsection*{Acknowledgements}
We warmly thank Federico Bambozzi, Marc Chapuis, Antoine Ducros, Florent Martin, Michel Matignon, Ariane Mézard, Johannes Nicaise, and Jérôme Poineau for interesting discussions about and beyond the content of this paper, and an anonymous referee for his comments and corrections.
We are especially grateful to Marc Chapuis for keeping us updated on the status of his related work \cite{Chapuis2017}.
During the preparation of this work, the first author's research was supported by the European Research Council (Starting Grant project ``Nonarcomp'' no.307856) and by the Agence Nationale de la Recherche (project ``Défigéo''), while the second author's research was supported by the Max Planck Institute for Mathematics and the European Research Council (Starting Grant project ``TOSSIBERG'' no.637027).
Finally, each of us warmly thanks the other author, without whom this paper would have been finished much earlier (but with a lot more mistakes).


\section{Semi-affinoid analytic spaces}
\label{section_semi-affinoids}

Let $R$ be a complete discrete valuation ring, $K$ its fraction field, $\pi$ a uniformizer of $R$, and $k=R/\pi R$ the residue field of $K$.
In this section we recall the notions of special $R$-algebras and their associated $K$-analytic spaces.
\medskip

A topological ring $A$ is called a \emph{noetherian adic ring} if $A$ is noetherian, separated and complete and there exists an ideal $J$ of $A$ such that the set of powers $\{J^\ell\}_{\ell>0}$ of $J$ is a fundamental system of neighborhoods of $0$ in $A$.
Any ideal $J$ as above is called an \emph{ideal of definition} of $A$.
Note that an ideal of definition of $A$ is not unique; for example if $J$ is an ideal of definition then so is $J^n$ for $n\geq1$.
However, there is a {largest ideal of definition} of $A$, the ideal generated by those elements of $A$ which are topologically nilpotent in $A$, i.e. nilpotent in $A/J$ for some (and thus for every) ideal of definition $J$ of $A$. 
A topological $R$-algebra $A$ is called a \emph{special $R$-algebra} if it is a noetherian adic ring and $A/J$ is of finite type over $k$ for some ideal of definition $J$ of $A$.

Recall that $R\{X_1,\ldots,X_m\}=\varprojlim_{\ell\geq1}\big(R/(\pi^\ell)\big)[X_1,\ldots,X_m]$ is the sub-algebra of the $R$-algebra $R[[X_1,\ldots,X_m]]$ consisting of those power series in the variables $(X_1,\ldots,X_m)$ whose coefficients tend to zero in a $\pi$-adic norm.
By \cite[1.2]{Ber96}, the special $R$-algebras are exactly the adic $R$-algebras of the form 
\[
\frac{R\{X_1,\ldots,X_m\}\lbrack\lbrack Y_1,\ldots,Y_n \rbrack\rbrack}{I}\cong \frac{R[[Y_1,\ldots,Y_n]]\{X_1,\ldots,X_m\}}{I},
\]
with ideal of definition generated by $\pi$ and by the $Y_i$'s.
Observe that all special $R$-algebras are excellent: this follows from \cite[Proposition~7]{Valabrega75} when the the characteristic of $K$ is positive and from \cite[Theorem~9]{Valabrega76} when it is zero.

An \emph{affine special formal $R$-scheme} $\X$ is the formal spectrum of a special $R$-algebra $A$, that is
\[
\X=\Spf(A)=\varinjlim_J\Spec\big(\bigslant{A}{J}\big),
\]
where the limit is taken over all the ideals of definition $J$ of $A$.
While this definition can be easily globalized, in this paper we content ourselves with the affine case, mostly as a convenient way to keep track of special $R$-algebras, and we do not make use of any deep results about formal schemes.
The interested reader can find a more thorough introduction to the theory of noetherian formal schemes in \cite{Bosch14}.
\medskip

We now recall how to associate a $K$-analytic space $\X^\beth$ with an affine special formal $R$-scheme $\X$.
This construction was introduced for rigid spaces by Berthelot in \cite{Berthelot}; we refer the reader to that paper as well as to \cite[\S7]{deJ95} and \cite{Ber96} for more detailed expositions.

If $\X$ is the formal spectrum of a special $R$-algebra of the form
$${R\{X_1,\ldots,X_m\}\lbrack\lbrack Y_1,\ldots,Y_n\rbrack\rbrack}/{(f_1,\ldots,f_r)},$$
then the associated $K$-analytic space is
$$\X^\beth = V(f_1,\ldots,f_r)\subset D^m_K\times_K (D^-_K)^{n}\subset\mathbb A^{m+n,\mathrm{an}}_K,$$
where 
$$D^m_K=\big\{x\in\mathbb A^{m,\mathrm{an}}_K\;\big|\;|X_i(x)|\leq1\text{ for all }i=1,\ldots,m\big\}$$
is the $m$-dimensional closed unit disc in $\mathbb A^{m,\mathrm{an}}_K$, 
$$(D^-_K)^{n}=\big\{x\in\mathbb A^{n,\mathrm{an}}_K\;\big|\;|Y_i(x)|<1\text{ for all }i=1,\ldots,n\big\}$$ 
is the $n$-dimensional open unit disc in $\mathbb A^{n,\mathrm{an}}_K$, and $V(f_1,\ldots,f_r)$ denotes the zero locus of the analytic functions $f_i$.

A more intrinsic definition of the analytic space $\X^\beth$ that also has the advantage of being clearly independent of the choice of a presentation of the special $R$-algebra $\calO_\X(\X)$ is the following.
Let $A$ be a special $R$-algebra and let $I$ be the largest ideal of definition of $A$.
For every $n>0$, denote by $A\left[I^n/\pi\right]$ the subring of $A\otimes_RK$ generated by $A$ and by the elements of the form $i/\pi$ for $i\in I^n$, and write $B_n$ for the $I$-adic completion of $A\left[I^n/\pi\right]$.
Finally, set $C_n=B_n\otimes_RK$.
Then the algebras $C_{n}$ are affinoid over $K$ and the canonical morphisms $C_{n+1}\to C_{n}$ identify $\calM(C_n)$ to an affinoid domain of $\calM(C_{n+1})$ and $(\Spf A)^\beth$ to the increasing union of the affinoid spaces $\calM(C_n)$.

\begin{remark}
The construction of $\X^\beth$ is functorial, sending an open immersion to an embedding of a closed subdomain, therefore it globalizes to general special formal $R$-schemes by gluing.
If $\X$ is of finite type over $R$ then $\X^\beth$ is compact, and this construction coincides with the classical one by Raynaud (see \cite{Raynaud74} or \cite[\S7.4]{Bosch14}).
\end{remark}

We say that a $K$-analytic space $X$ is \emph{semi-affinoid} if it is of the form $\X^\beth$ for some affine special formal $R$-scheme $\X$, and call \emph{model} of $X$ any flat affine special formal $R$-scheme whose associated $K$-analytic space is isomorphic to $X$.
The terminology semi-affinoid $K$-analytic space is used in \cite{Martin2017}; those should not be confused with the semi-affinoid $K$-spaces from \cite{Kappen12}.
Every (strictly) affinoid $K$-analytic space is semi-affinoid, admitting a model which is of finite type over $R$. 

Let $\X=\Spf(A)$ be a flat affine special formal $R$-scheme.
By \cite[7.1.9]{deJ95}, the $K$-analytic space $\X^\beth$ depends only on the $K$-algebra $A\otimes_RK$ (such a $K$-algebra is called semi-affinoid in \cite{Kappen12}).
In particular, if $\Spf(B)$ is another model of $\X^\beth$ then $B\otimes_RK \cong A\otimes_RK$, and therefore, since $A\otimes_RK\cong (A/\pi\text{-torsion})\otimes_RK$, every semi-affinoid $K$-analytic space has a model.
Moreover, $\X^\beth$ does not change if we replace $\X$ by its integral closure in the generic fiber, that is the affine special formal $R$-scheme $\Spf(B)$, where $B$ is the integral closure of $A$ in $A\otimes_RK$.
Observe that $B$ is a special $R$-algebra since it is finite over $A$ because $A$ is excellent.
\medskip

If $\X=\Spf(A)$ is a model of a semi-affinoid space $X$, then the canonical homomorphism $A \otimes_R K \to \calO_X(X)$ is injective.
Indeed, let $f$ be an element of $A \otimes_R K$ which vanishes in $\calO_X(X)$, and let $\frakM$ be a maximal ideal of $A \otimes_R K$.
By \cite[Lemma~7.1.9]{deJ95} $\frakM$ corresponds to a point $x$ of $X$, and the image $f(x)$ of $f$ in the completed local ring of $X$ at $x$ coincides with the image $\alpha(f_\frakM)$ via the completion morphism $\alpha\colon (A \otimes_R K)_\frakM \to (A \otimes_R K)_\frakM^\wedge$ of the image $f_\frakM$ of $f$ in the localization of $A \otimes_R K$ at $\frakM$.
It follows that $\alpha(f_\frakM)=0$, hence $f_\frakM=0$ because $(A \otimes_R K)_\frakM$ is a local noetherian ring and so its completion morphism is injective.
Since this is true for every maximal ideal of $A \otimes_R K$, it follows that $f=0$.
Moreover, since $A$ is flat over $R$, the canonical homomorphism $A \to A \otimes_R K$ is also injective.
This shows that if $X$ is reduced then both $A$ and $A \otimes_R K$ are reduced.
Since by \cite[Proposition~7.2.4.c]{deJ95} $X$ is reduced whenever $A$ is reduced, these three properties are actually equivalent.

A reduced semi-affinoid space is completely determined by the ring $\OOO_X(X)$ of \emph{bounded functions}, which is defined as the subring of $\calO_X(X)$ consisting of those analytic functions $f$ such that $|f(x)|\leq1$ for every point $x$ of $X$, as is explained in the following lemma.

\begin{lemma}\label{lemma_canonical_model}
If $X$ is a reduced semi-affinoid $K$-analytic space then $\X=\Spf\big({\OOO_X(X)}\big)$ is a model of $X$.
Moreover $\X$ is the unique model of $X$ which is integrally closed in its generic fiber.
\end{lemma}
\begin{proof}
Let $\Spf(A)$ be a model of $X$.
By replacing $A$ with its integral closure in $A\otimes_R K$ (which is itself special because $A$ is excellent) we can assume that $A$ is integrally closed in its generic fiber.
Moreover, $A$ is reduced because $X$ is reduced.
Therefore, by \cite[Theorem 2.1]{Martin2017} we have $A\cong\OOO_X(X)$.
Observe that $\OOO_X(X)$ is flat over $R$ because the analytic function induced by $\pi$ on $X$ becomes invertible in $\OO_X(X)$.
This proves both parts of our statement.
\end{proof}

If $X=(\Spf A)^\beth$ as in the proof above is not reduced, then the canonical injection $A\to\OOO_X(X)$, which is an isomorphism after killing the nilradicals, may fail to be surjective in general, and $\OOO_X(X)$ may then fail to be special over $R$. See \cite[Example 2.3]{Martin2017} for an example.

Let $X$ be a reduced semi-affinoid $K$-analytic space. 
We call \emph{canonical model} of $X$ the special formal $R$-scheme $\X=\Spf\big({\OOO_X(X)}\big)$, and we define the \emph{canonical reduction} $X_0$ of $X$ to be the reduced affine special formal $k$-scheme $(\X_s)_{red}$ associated with the special fiber $\X_s=\X\otimes_{R} k$ of $\X$.
We say that a semi-affinoid $K$-analytic space $X$ is \emph{distinguished} if it is reduced and the special fiber $\X_s$ of its canonical model is already reduced, i.e. if it coincides with $X_0$.

\begin{remark}
Let $X$ be an affinoid $K$-analytic space.
Then $X$ is distinguished if and only if its affinoid algebra $\calA$ is distinguished in the classical sense, see \cite[\S6.4.3]{BGR}.
Moreover, when this is the case then the canonical reduction $X_0$ of $X$ is the usual reduction 
 of the affinoid space $X$.
Indeed, $\calA$ is distinguished if and only if it is reduced and its spectral norm $|\cdot|_{\sup}$ takes values in $|K|$ (this is \cite[6.4.3/1]{BGR}, since by \cite[3.6]{BGR} discretely valued fields are stable).
In particular, to show both implications we can assume that $X$ is reduced, and therefore by Lemma~\ref{lemma_canonical_model} its canonical model is $\X=\Spf\calA^\circ$, where $\calA^\circ=\big\{f\in\calA\,\big|\,|f|_{\sup}\leq1\big\}$.
Observe that $|\calA|_{\sup}=|K|$ if and only if $\pi\calA^\circ=\big\{f\in\calA\,\big|\,|f|_{\sup}\leq|\pi|\big\}$ coincides with $\calA^{\circ\circ}=\big\{f\in\calA\,\big|\,|f|_{\sup}<1\big\}$.
It follows that if $\calA$ is distinguished then the special fiber $\X_s=\Spec\big(\calA^\circ/\pi\calA^{\circ}\big)$ of $\X$ coincides with the usual reduction $\Spec\big(\calA^\circ/\calA^{\circ\circ}\big)$ of $X$, which in particular proves that $X$ is distinguished.
Conversely, assume that $|\calA|_{\sup}\supsetneq|K|$, so that $\calA$ is not distinguished.
Then there exists an element $f\in \calA$ such that $|\pi|<|f|_{\sup}<1$.
This implies that there exists some $n>0$ such that $f^n\in\pi\calA^\circ$ while $f\in \calA^\circ\setminus \pi\calA^\circ$, hence $\X_s$ is not reduced, which proves that $X$ is not distinguished.
\end{remark}

We conclude the section with a simple lemma.

\begin{lemma}\label{lemma_reduced_special_fiber_implies_canonical_model}
Let $X$ be a semi-affinoid $K$-analytic space and assume that $X$ has a model $\X$ with reduced special fiber.
Then $X$ is distinguished and $\X$ is its canonical model.
\end{lemma}

\begin{proof}
Since $\X$ is $R$-flat and has reduced special fiber then it is reduced, so $X$ is reduced as well.
By Lemma~\ref{lemma_canonical_model} we have to show that $\X$ is integrally closed in its generic fiber. 
Write $\X=\Spf(A)$ and let $\alpha$ be an element of $A\otimes_RK=A[\pi^{-1}]$ such that $\alpha^m+a_{m-1}\alpha^{m-1}+\ldots+a_0=0$, where the $a_i$ are elements of $A$.
We can write $\alpha=a\pi^r$, where $a$ is an element of $A\setminus \pi A$ and $r\in\mathbb Z$.
If $r<0$, since $A$ has no $\pi$-torsion we can divide both sides of the equality above by $\pi^{mr}$ to deduce that $a\in \pi A$, which gives a contradiction.
Therefore $r\geq0$, hence $\alpha\in A$, which is what we wanted to prove.
\end{proof}


\section{Weil restrictions}
\label{section_weil}

In this section we introduce the Weil restriction functor and discuss its representability for semi-affinoid analytic spaces and their models.
We then define the dilated Weil restriction, a modified version of the usual Weil restriction for formal schemes, as a notational tool that is useful to describe in a simple way the Weil restriction of a semi-affinoid space.
\medskip

Let $\calC$ be a category with fiber products.  
For an object $\s$ of $\calC$ denote by $\calC_\s$ the category of $\s$-objects of $\calC$ and let $\s' \to \s$ and $\X \to \s'$ be two morphisms in $\calC$. 
The \emph{Weil restriction} of $\X$ to $\s$ is defined as the functor 
\begin{alignat*}{2}
\prod\nolimits_{\s'|\s}\X \colon && \calC_\s & \longrightarrow  (Sets) \\
					  && (\T\to \s) & \longmapsto  \Hom_{\s'}(\T\times_\s\s',\X).
\end{alignat*}
When this functor is representable we also denote by $\prod_{\s'|\s}\X$ the object of $\calC_\s$ which represents it, and call it \emph{Weil restriction} of $\X$ to $\s$. In this paper we only deal with Weil restrictions of affine special formal schemes with respect to a finite extension $R'|R$ of complete discrete valuation rings, and with Weil restrictions of semi-affinoid analytic spaces with respect to a finite separable extension $K'|K$ of complete discrete valuation fields. 
In these settings, we write $\prod\nolimits_{R'|R}$ instead of $\prod\nolimits_{\Spf R'|\Spf R}$ and $\prod\nolimits_{K'|K}$ instead of $\prod\nolimits_{\mathcal M(K')|\mathcal M(K)}$ respectively.

Observe that it follows from the definition that the Weil restriction is compatible with base change, that is, for every morphism $\T\to \s$ in $\calC$ we have
\[
\prod\nolimits_{\T'|\T}\big(\X\times_{\s'}\T'\big)=\Big(\prod\nolimits_{\s'|\s}\X\Big)\times_\s\T,
\]
where $\T'=\T\times_\s\s'$.

When $\calC$ is the category of schemes, the representability of the Weil restriction functor is well understood: we refer the reader to \cite[\S 7.6]{NeronModels} for a thorough discussion.

As observed in \cite{Bertapelle00}, the representability of the Weil restriction functor on the category of formal $R$-schemes can be studied in terms of the representability of the Weil restriction of ordinary schemes in the following way.
If $\X=\varinjlim X_n$ is a formal scheme over $\s'=\varinjlim \s'_n$, where each $X_n$ is a scheme over $\s'_n$ such that $\prod_{\s'_n|\s_n}X_n$ is representable for all $n$, then $\prod_{\s'|\s}\X$ is represented by $\varinjlim\prod_{\s'_n|\s_n}X_n$.
See Theorem~1.4 of \emph{loc. cit.} for a precise statement including conditions for the representability. 
In the case of an affine special formal scheme the Weil restriction is always representable, and the formal scheme representing it can be described explicitly in a simple way; this is the content of the next lemma. 
Recall that an element of a noetherian adic ring $A$ is said to be topologically nilpotent if its powers converge to zero in $A$, or equivalently if it belongs to some ideal of definition of $A$.

\begin{lemma}\label{lemma_restriction_affine}
Let $A$ and $A'$ be two special $R$-algebras.
Assume that $A'$ is a free $A$-module with basis $e_0,\ldots,e_{m-1}$ and that $e_0,\ldots,e_d$ are the elements of this basis which are topologically nilpotent in $A'$.
Let $\calA'=A'\{\underline X\}[[\underline Y]]/I$ be a special $A'$-algebra, where $\underline X$ is an $r$-tuple and $\underline Y$ is an $s$-tuple of variables.
Then $\prod_{\Spf A'|\Spf A}\Spf \calA'$ is represented by the formal spectrum of the special $A$-algebra
\[
\frac{A\{\underline X_{0},\ldots,\underline X_{m-1},\underline Y_{0},\ldots,\underline Y_{d}\}[[\underline Y_{d+1},\ldots,\underline Y_{m-1}]]}{I^{\mathfrak{c}}},
\]
where the $\underline X_i$ are $r$-tuples and the $\underline Y_j$ are $s$-tuples of variables, and $I^\mathfrak{c}$ is the \emph{ideal of the coefficients} of $I$, that is the ideal generated by the coefficients appearing when expressing the elements of $I$ in the basis $e_0,\ldots,e_{m-1}$ via the isomorphism $A'\cong e_0A\oplus\ldots\oplus e_{m-1}A$ and writing $T=T_0e_0+\ldots+T_{m-1}e_{m-1}$ for every variable $T$ among the $\underline X$ and $\underline Y$.
\end{lemma}

\begin{proof}
The proof is a slight extension of the one of \cite[\S 7.6, Theorem 4]{NeronModels}.
The only difficulty which is not present in \emph{loc. cit.} is the case $\calA'=A'[[Y]]$, where $Y$ is a single variable.
Observe that being covered by affine formal $R$-schemes, any formal scheme is determined by its functor of points on affine formal $R$-schemes, that is by its points with values in adic $R$-algebras.
Therefore, what we need to show is that for any adic $A$-algebra $L$ there is a bijection
\[
\Hom_{A'}\big(A'[[Y]],L\hat\otimes_AA'\big)\cong\Hom_A\big(A\{Y_0,\ldots,Y_d\}[[Y_{d+1},\ldots,Y_{m-1}]],L\big)
\]
which is functorial in $L$, where the $\Hom$ sets are sets of continuous homomorphisms.
Since $A'$ is finite over $A$ we have $L\hat\otimes_AA' \cong L\otimes_AA'$ by \cite[3.7.3/6]{BGR}.
A continuous $A'$-homomorphism $\sigma'\colon A'[[Y]]\to L\otimes_AA'$ is determined by $\sigma'(Y)$, which is a topologically nilpotent element of $L\otimes_AA'$. 
Using the decomposition
\[
L\otimes_AA'\cong\bigoplus_{i=0}^{m-1}L e_i
\]
we write
\[
\sigma'(Y)=\sum_{i=0}^{m-1}\widetilde\sigma(Y_i)\otimes e_i,
\]
where the $\widetilde\sigma(Y_i)$ are elements of $L$.
These elements give rise to a homomorphism of $A$-algebras $\widetilde\sigma \colon A[Y_0,\ldots,Y_{m-1}]\to L$.
The fact that $\sigma'(Y)$ is topologically nilpotent translates to the fact that $\widetilde\sigma(Y_i)\otimes e_i$ is topologically nilpotent in $L\otimes_AA'$ for each $i$.
This imposes no condition on $\widetilde\sigma(Y_0),\ldots,\widetilde\sigma(Y_d)$, while $\widetilde\sigma(Y_{d+1}),\ldots,\widetilde\sigma(Y_{m-1})$ have to be topologically nilpotent in $L$.
Therefore $\widetilde\sigma$ extends to a continuous morphism of $A$-algebras $\sigma \colon A\{Y_0,\ldots,Y_d\}[[Y_{d+1},\ldots,Y_{m-1}]]\to L$.
The association $\sigma'\mapsto\sigma$ is the bijection we are after, and the general case is then a simple translation of the arguments of \cite[\S 7.6, Proposition 2 and Theorem 4]{NeronModels}.
\end{proof}

In the remaining part of the section we discuss Weil restrictions in the context of $K$-analytic spaces, building on the work of Bertapelle \cite{Bertapelle00}.

Let $K'$ be a finite and separable non-archimedean extension of $K$.
As is the case for affine schemes and formal schemes, Weil restrictions of affinoid $K'$-analytic spaces to $K$ are always representable.
However, the argument of Lemma~\ref{lemma_restriction_affine} needs to be refined, since if $\sigma'\colon K'\{X\}\to K'$ is any bounded homomorphism of $K'$-affinoid algebras then the induced homomorphism $\sigma\colon K\{X_0\ldots,X_{m-1}\}\to K$ may not be bounded.
Indeed, one can have $||\sigma'(X)||_\mathrm{sup}=||\sum_{i=0}^{m-1}\sigma(X_i)\otimes e_i||_\mathrm{sup}\leq 1$ but $||\sigma(X_i)||_\mathrm{sup} > 1$ for some $i$.

To control the sup norm of $\sigma'(X)$ it is useful to consider characteristic polynomials.
Let $\calA=K'\{X_1,\ldots,X_n\}/I$ be a $K$-strictly affinoid algebra.
Recall that, for each $n$-tuple of positive real numbers ${r} = (r_1, \dots, r_n)$, the $K$-algebra $K\{r_1^{-1}X_1, \dots, r_n^{-1}X_n\}:= \{\sum_I a_I X^I | |a_I|r^I \to 0 \}$ is the algebra of functions that converge on the closed $K$-polydisc of polyradius ${r}$.
By \cite[Proposition 1.8]{Bertapelle00}, we have
\(
\prod\nolimits_{K'|K}\calM(\calA) \cong \varinjlim_{\lambda\in \N} \calM \left( \calB_\lambda \right),
\)
where $\calB_\lambda$ is the strictly $K$-affinoid algebra
\begin{equation}\label{equazione_restrizione_affinoide}
\resizebox{1\textwidth}{!}{$
	\calB_\lambda
	=
	\frac{K\big\{|\pi|^\lambda X_{0,1},\ldots,|\pi|^\lambda X_{0,n},|\pi|^\lambda X_{1,1},\ldots,|\pi|^\lambda X_{m-1,n}\big\}\big\{Z_{0,1}, \dots, Z_{m-1,n}\big\}}{\big(I^\mathfrak{c}, Z_{0,1}-c_0(X_{\bullet,1}), \ldots, Z_{m-1,n}-c_{m-1}(X_{\bullet,n})\big)} ,
$}
\end{equation}
$c_j\in K[x_0,\ldots,x_{m-1}]$ is the coefficient of the term of degree $j$ of the characteristic polynomial of the endomorphism given by the multiplication by $\sum_{i=0}^{m-1} x_ie_i$ in the $K[x_0\ldots,x_{m-1}]$-vector space $K'[x_0\ldots,x_{m-1}]$, and as before $I^\mathfrak{c}$ is the ideal of the coefficients of $I$.
Moreover, since $K'|K$ is separable, the Weil restriction $\prod\nolimits_{K'|K}\calM(\calA)$ is compact, therefore there exists a $\lambda_0 \in \N$ such that for every $\lambda>\lambda_0$ we have 
\(
\prod\nolimits_{K'|K}\calM(\calA) \cong \calM \left( \calB_\lambda \right).
\)
In other words, $\prod\nolimits_{K'|K}\calM(\calA)$ is the affinoid subspace of $(\mathbb A^{nm}_K)^\mathrm{an}$ cut out by the elements of $I^\mathfrak{c}$ and by the conditions $|c_j(X_{\bullet,i})|\leq 1$.

\begin{example}\label{esempio_disco_gtt}
Let $p \neq 2$ be a prime number, set $R=\Z_p$ and $R'=\Z_p[\sqrt{p}]$, and let $K$ and $K'$ be the respective fraction fields. 
Then the set $\{1, \sqrt{p}\}$ is a basis both for $R'$ over $R$ and for $K'$ over $K$, and therefore Lemma \ref{lemma_restriction_affine} gives $\prod_{R'|R}\Spf R'\{X\} \cong \Spf \big(R\{X_0, X_1\}\big)$.
On the other hand, by the equation (\ref{equazione_restrizione_affinoide}) the Weil restriction of the unit disc is given by
\[ 
\prod\nolimits_{K'|K}\calM(K'\{X\})\cong \calM\big(K\{|p|^\lambda X_0, |p|^\lambda X_1\}\{ -2X_0, X_0^2-pX_1^2\}\big),
\]
which, for $\lambda$ big enough, is the analytic subspace of $\A^{2, an}_K$ defined by the inequalities
\[
\begin{cases} 
|X_0|\leq 1 \\ |X_0^2-pX_1^2|\leq 1, 
\end{cases}
\] 
which is isomorphic to the 2-dimensional polydisc of polyradius $\big(1, |p|^{-\sfrac{1}{2}} \big)$.
Indeed, this polydisc is clearly contained in $\prod\nolimits_{K'|K}\calM(K'\{X\})$.
For the converse inclusion, observe that if $|pX_1^2|=|X_0^2|$ then $|pX_1^2|\leq 1$, while if $|pX_1^2|\neq|X_0^2|$ then $|X_0^2-pX_1^2|= \max\big\{|X_0|^2, |p||X_1|^2\big\}\leq 1$.
In both cases, $|X_1|\leq {1}/{|\sqrt{p}|}$.
\end{example}

Example~\ref{esempio_disco_gtt} shows that $\prod\nolimits_{K'|K}\big((\Spf R'\{X\})^\beth\big)$ and $\big(\prod\nolimits_{R'|R}\Spf R'\{X\}\big)^\beth$ do not coincide in general. 
It is therefore convenient to introduce a variant of the formal Weil restriction to be able to describe the Weil restriction of a semi-affinoid $K'$-analytic space in terms of an $R$-model.
In order to do this, let us fix some notation.
Given two coprime positive integers $a$ and $b$, we set $c=b\lceil a/b \rceil -a\geq0$, where ${\lceil a/b \rceil}$ is the ceiling of $a/b$, and we write 
\begin{equation}\label{notation_fractional_special_algebras}
R\big\{|\pi|^{\sfrac{a}{b}}X\big\}=\frac{R\{W,Z\}}{(\pi^c Z-W^b)}.
\end{equation}
This way, after a base change to $K$ we obtain 	
\begin{align*}
R\big\{|\pi|^{\sfrac{a}{b}}X\big\}\otimes_RK & \cong K\{W,Z\}/(\pi^c Z-W^b) \cong K\big\{|\pi|^{{\lceil {\sfrac{a}{b}} \rceil}}X, Z\big\}/(Z-\pi^{a}X^b) 
\\
& \cong K\big\{|\pi|^{\sfrac{a}{b}}X\big\},
\end{align*}
which is the $K$-affinoid algebra of the closed disc of radius $|\pi|^{-\sfrac{a}{b}}$, where the second isomorphism is defined by sending $W$ to $\pi^{{\lceil {\sfrac{a}{b}} \rceil}}X$.
Similarly, we set $R\big[\big[\,|\pi|^{\sfrac{a}{b}}X\big]\big]=R[[W,Z]]/(\pi^c Z-W^b)$.
This notation is primarily a way of keeping track of equations when considering models of polydiscs of rational polyradii that are contained in a given polydisc.
For example, for $n\in\N$ then $R\big\{|\pi|^{n}X\big\}$ is isomorphic to the algebra $R\{W\}$ of a closed disc of radius one, but this notation allows us to keep equations in the variable $X$, avoiding a change of variables.
More interesting examples with rational radii are given in section \ref{section_fractional_annuli}.

Let $A'=R'\{\underline X\}[[\underline Y]]/I$ be a special $R'$-algebra, where $\underline X$ is an $r$-tuple and $\underline Y$ is an $s$-tuple of variables, and set $\X = \Spf{A'}$. 
We define the \emph{dilated Weil restriction} $\prod_{R'|R}^\mathrm{dil}\X$ of $\X$ to $R$ to be the affine special $R$-formal scheme $\varinjlim_{\lambda\in \N} \Spf(A_\lambda)$, with $A_\lambda$ the special $R$-algebra
\[
\resizebox{1\textwidth}{!}{$A_\lambda = \frac{R\big\{|\pi|^\lambda \underline X_0,\ldots,|\pi|^\lambda\underline X_{m-1},|\pi|^\lambda\underline Y_{0},\ldots,|\pi|^\lambda\underline Y_{m-1}\big\}\big\{Z_{0,1}, \dots, Z_{m-1,r}\big\}\big[\big[\,Z_{0,r+1}, \dots, Z_{m-1,r+s}\big]\big]}{\big (I^{\mathfrak{c}},Z_{0,1}-c_0(X_{\bullet,1}), \ldots, Z_{m-1,r}-c_{m-1}(X_{\bullet,r}), Z_{0,r+1}-c_0(Y_{\bullet,1}), \ldots, Z_{m-1,r+s}-c_{m-1}(Y_{\bullet,s})\big)},
$}
\]

where, as before, $c_j$ is the coefficient of the term of degree $j$ in the characteristic polynomial of $\sum_{i=0}^{m-1} x_ie_i\in K'[x_0\ldots,x_{m-1}]$, and $I^\mathfrak{c}$ is the ideal of the coefficients of $I$. 
Note that, since $K'|K$ is separable, \cite[Proposition 1.8]{Bertapelle00} ensures that the $K$-analytic space associated with the dilated Weil restriction is contained in a compact subspace of the analytic affine space, and therefore there exists a positive integer $\lambda_0$ such that $\prod_{R'|R}^\mathrm{dil}\X= \Spf A_\lambda$ for every $\lambda \geq \lambda_0$.

The following proposition shows that, as expected, the dilated Weil restriction of an affine special formal scheme is a model of the Weil restriction of the associated analytic space.

\begin{proposition}\label{proposition_semi-affinoid_restriction}
Let $K'|K$ be a finite extension of discretely valued fields, let $R'|R$ be the corresponding extension of valuation rings, and let $\X$ be an affine special formal $R'$-scheme.
Then  we have 
\[
\prod\nolimits_{K'|K} \big( \X^\beth \big)
 \cong 
\Big(\prod\nolimits_{R'|R}^\mathrm{dil}\X\Big)^\beth.
\]
\end{proposition}
\begin{proof}
Write $\X=\Spf\big(R'\{\underline X\}[[\underline Y]]/I\big)$, so that $\X^\beth$ is the increasing union of the affinoid domains $U_n=\calM\big(K'\{\underline X, |\pi|^{-\sfrac{1}{n}} \underline Y\}/I\big)$ for $n>0$.
Let $\frakU$ be the set of all strictly $K'$-affinoid domains of $\X^\beth$, let $Y$ be a strictly $K$-affinoid space, and let $\phi:Y_{K'} \to \X^\beth$ be a $K'$-analytic morphism.
We claim that the covering $\{\phi^{-1}(U) \,|\, U\in\frakU\}$ of $Y_{K'}$ can be refined to a covering of the form $\{V \times_K K' \,|\, V\in \frakV \}$ for some admissible covering $\frakV$ of $Y$. 
Indeed, since $\X^\beth$ is semi-affinoid, it is an increasing union of affinoid domains, and so $\phi(Y_{K'})$, being compact, is contained in an affinoid domain $W$ of $\X^\beth$.
The claim follows by applying \cite[Proposition 1.14]{Bertapelle00} to $W$, since $\prod\nolimits_{K'|K}W$ is representable. 
Now, by \cite[Theorem 1.13]{Bertapelle00}, we deduce that $\prod\nolimits_{K'|K}\X^\beth$ is obtained by gluing $\prod\nolimits_{K'|K}U$ for $U\in \frakU$, and therefore also by gluing the $\prod\nolimits_{K'|K}U_n$ along the immersions $\prod\nolimits_{K'|K}U_n \hookrightarrow \prod\nolimits_{K'|K}U_{n'}$ whenever $n'>n$.
When $\X^\beth$ is affinoid, that is when $\X$ is of finite type over $R'$, the proposition follows immediately from the description of the equation~\eqref{equazione_restrizione_affinoide}, therefore to prove the theorem it is enough to treat the case of the open disc.
So assume from now on that $\X=\Spf R'[[Y]]$, so that $U_n=\calM\big(K'\{|\pi|^{-\sfrac{1}{n}} Y\}\big)$ is a closed disc of radius $|\pi|^{1/n}$.
Weil restrictions of general $K$-affinoid spaces, including non-strict ones, behave exactly in the same way as in \ref{equazione_restrizione_affinoide}, see \cite{Wahle09}.
In particular, by Proposition 3.1.4 of \emph{loc. cit.}, we have 
\[
\resizebox{1\textwidth}{!}{$
	\prod\nolimits_{K'|K}U_n
	\cong 
	\bigcup_{\lambda>0}\calM\bigg( \frac{K\big\{|\pi|^\lambda Y_{0},\ldots,|\pi|^\lambda Y_{m-1}\big\}\big\{|\pi|^{-\sfrac{1}{n}} Z_{0}, \dots, |\pi|^{-\sfrac{(m-1)}{n}} Z_{m-1}\big\}}{\big(Z_{0}-c_0(Y_{\bullet}), \ldots, Z_{m-1}-c_{m-1}(Y_{\bullet})\big)} \bigg)
	$}
\]
as subspaces of $(\mathbb A^m_K)^\mathrm{an}$.
It follows that 
\begin{align*}
\prod\nolimits_{K'|K}\X^\beth & \cong \bigcup_n \prod\nolimits_{K'|K}U_n 
\\
& \cong 
	 \bigcup_{n,\lambda} 
	 \resizebox{.7675\textwidth}{!}{$\calM\bigg( \frac{K\big\{|\pi|^\lambda Y_{0},\ldots,|\pi|^\lambda Y_{m-1}\big\}\big\{|\pi|^{-\sfrac{1}{n}} Z_{0}, \dots, |\pi|^{-\sfrac{(m-1)}{n}} Z_{m-1}\big\}}{\big(Z_{0}-c_0(Y_{\bullet}), \ldots, Z_{m-1}-c_{m-1}(Y_{\bullet})\big)} \bigg)
	$}
\\
& \cong \bigcup_{\lambda} \Spf\bigg( \frac{R\big\{|\pi|^\lambda Y_{0},\ldots,|\pi|^\lambda Y_{m-1}\big\}\big[\big[Z_{0}, \dots, Z_{m-1}\big]\big]}{\big(Z_{0}-c_0(Y_{\bullet}), \ldots, Z_{m-1}-c_{m-1}(Y_{\bullet})\big)} \bigg).
\end{align*}
This is the subspace of $(\mathbb A^m_K)^\mathrm{an}$ defined by the inequalities $|c_j(Y_{\bullet})|<1$, which is precisely $\big(\prod\nolimits_{R'|R}^\mathrm{dil}\X\big)^\beth$ from the definition of the dilated Weil restriction.
\end{proof}

In the case of unramified extensions, the usual formal Weil restriction is sufficient to obtain a model of the Weil restriction of a semi-affinoid $K'$-analytic space, as the following proposition ensures.
Observe that in this case no element of a basis of $R'$ over $R$ is topologically nilpotent.

\begin{proposition}\label{proposition_commutativity_unramified}
Let $K'|K$ be a finite extension of discretely valued fields, let $R'|R$ be the corresponding extension of valuation rings, and let $\X$ be an affine special formal $R'$-scheme.
Then we have an immersion 
\[
\Big (\prod\nolimits_{R'|R}\X \Big)^\beth \longhookrightarrow \prod\nolimits_{K'|K}\X^\beth,
\]
which is an isomorphism if $K'$ is unramified over $K$.
\end{proposition}
\begin{proof}
When $\X$ is topologically of finite type, this result is proved in \cite[Prop. 2.5 (5)]{NicaiseSebag2008}. 
The general case then follows then from the description of semi-affinoid spaces as increasing unions of affinoid domains.
\end{proof}

In Example~\ref{esempio_disco_gtt} we computed explicitly the Weil restriction of a closed disc. 
Now we generalize that description to the restriction of any semi-affinoid space with respect to a tamely ramified extension.

\begin{proposition}\label{proposition_dilated_restriction_explicit}
Let $K'|K$ be a totally tamely ramified degree $m$ extension of discretely valued fields, let $R'|R$ be the corresponding extension of valuation rings, and let $A'=R'\{\underline X\}[[\underline Y]]/I$ be a special $R'$-algebra.
Then we have 
\[
\resizebox{1\textwidth}{!}{$
\Big(\prod\nolimits_{R'|R}^\mathrm{dil}\Spf A' \Big)^\beth
 \, \cong \, 
\bigg( \Spf \frac{R\{\underline X_0,|\pi|^{\frac{1}{m}} \underline X_1 \ldots, \vert\pi\vert^{\frac{m-1}{m}} \underline X_{m-1}\}[[\underline Y_0, |\pi|^{\frac{1}{m}} \underline Y_{1},\ldots,|\pi|^{\frac{m-1}{m}}\underline Y_{m-1}]]}{I^{\mathfrak{c}}} \bigg)^\beth.
$}
\]
\end{proposition}

\begin{proof}
We prove the result in the case of the open disc, that is $A'=R'[[Y]]$, the general case being completely analogous.
Let $\varpi$ be a uniformizer of $R'$ such $\varpi^m$ is a uniformizer of $R$. 
Then $\{1, \varpi, \ldots, \varpi^{m-1}\}$ is a basis of $R'$ over $R$.
To compute the coefficients $c_j$ appearing in the expression \ref{equazione_restrizione_affinoide} we consider the matrix of the multiplication by $\sum_{i=0}^{m-1} x_i\varpi^{i-1}$ in $K'[x_0, \dots, x_{m-1}]$, which is 
\[
M= \left (\begin{matrix} x_0 & \varpi^m x_{m-1} & \ldots & \varpi^m x_1 \\ x_1 & x_0 & \ldots & \varpi^m x_2 \\ \vdots & \vdots &\ddots & \vdots \\ x_{m-1} & x_{m-2} & \ldots & x_0\end{matrix} \right ).
\]
The $c_j$ correspond then to the sums of the ($m-j$)-th principal minors of $M$. 
Now, for every $i \in \{0, \ldots, m-1\}$, set $y_i = \varpi^i x_i$. 
With this substitution, the matrix becomes 
\[
\left (\begin{matrix} y_0 & \varpi y_{m-1} & \ldots & \varpi^{m-1} y_1 \\ \frac{y_1}{\varpi} & y_0 & \ldots & \varpi^{m-2} y_2 \\ \vdots & \vdots &\ddots & \vdots \\ \frac{y_{m-1}}{\varpi^{m-1}} & \frac{y_{m-2}}{\varpi^{m-2}} & \ldots & y_0\end{matrix} \right ) \sim \left (\begin{matrix} y_0 & y_{m-1} & \ldots & y_1 \\ y_1 & y_0 & \ldots & y_2 \\ \vdots & \vdots &\ddots & \vdots \\ y_{m-1} & y_{m-2} & \ldots & y_0\end{matrix} \right ),
\]
the equivalent matrix on the right being obtained by multiplying the $i$-th row by $\varpi^i$ and the $j$-th column by $\varpi^{-j}$. 
Observe that the $c_j$ are invariant under equivalence, and the matrix obtained on the right is the matrix associated with the multiplication by $\sum_{i=0}^{m-1} y_i\alpha^i $ in $K(\alpha)[y_0, \ldots, y_{m-1}]$, where $\alpha^m=1$. 
Since the extension $K(\alpha)$ is unramified over $K$, the inequalities $|c_j(y_0, \ldots, y_{m-1})|< 1$ define an open unit disc. 
Therefore, over $K$ the inequalities $|c_j(x_0, \ldots, x_{m-1})|< 1$ define a polydisc of polyradius $\{1, {|\varpi|^{-1}}, \dots, {|\varpi|^{1-m}}\}$, which implies the result we want from the definition of the dilated Weil restriction.
\end{proof}

More generally, since any tamely ramified extension decomposes as a totally tamely ramified extension of an unramified extension, one can combine Proposition~\ref{proposition_commutativity_unramified} and Proposition~\ref{proposition_dilated_restriction_explicit} to compute in two steps the Weil restriction of a semi-affinoid space with respect to such an extension.
A computation of this kind is performed in Section~\ref{section_forms_annuli} to study tame forms of annuli.



\section{Fixed loci}
\label{section_fixed_locus}

We now move to the study of Galois-fixed loci for semi-affinoid analytic spaces and their models.
This is the second ingredient that we need to study forms of semi-affinoid $K'$-analytic spaces via their models.
\medskip

Let $\calC$ be a category, fix an object $\s$ of $\calC$, and let $\X \to \s$ be a morphism in $\calC$.
Let $G$ be a finite group acting equivariantly on $\X\to \s$, with the trivial action on $\s$.
Then the \emph{$G$-fixed locus} $\X^G$ of $\X$ is defined as the functor
\begin{alignat*}{2}
\X^G \colon && \calC_\s & \longrightarrow  (Sets) \\
					  && (\T\to \s) & \longmapsto  \big(\Hom_{\s}(\T,\X)\big)^G,
\end{alignat*}
where $\big(\Hom_{\s}(\T,\X)\big)^G$ is the subset of $\Hom_{\s}(\T,\X)$ consisting of those $\s$-morphisms $f$ such that $g\circ f=f$ for every $g$ in $G$.
When this functor is representable we also denote by $\X^G$ the object of $\calC_\s$ which represents it, and call it the \emph{$G$-fixed locus} of $\X$.

Observe that the $G$-fixed locus is compatible with base change: if $G$ acts trivially on a morphism $\T\to \s$ we have
\[
(\X\times_\s\T)^G\cong\X^G\times_\s\T.
\]
In this paper we only consider fixed loci of affine special formal schemes and semi-affinoid analytic spaces.
We start by discussing the case of special formal $R$-schemes as it is slightly simpler: if $\X$ is a separated formal $R$-scheme then $\X^G$ is a closed formal subscheme of $\X$.
Indeed, write $\X=\varinjlim X_n$, then $G$ acts on each of the $X_n$ and they are all separated, so by \cite[Proposition 3.1]{Edixhoven92} the $G$-fixed locus $(X_n)^G$ of $X_n$ is represented by a closed subscheme of $X_n$.
Then $\X^G$ is the closed formal subscheme $\varinjlim (X_n)^G$ of $\X$.

The fixed loci of affine formal $R$-schemes can be simply described as follows. 
If $G$ is a finite group acting continuously on a commutative ring $A$, we define the \emph{ring of $G$-coinvariants} $A_G$ of $A$ as the quotient of $A$ by the ideal generated by the set $\{a-g(a)\,|\,a\in A, g\in G\}$.
Observe that if $A$ is a special $R$-algebra then $A_G$, being an adic quotient of $A$, is itself a special $R$-algebra.

\begin{lemma}\label{lemma_representability_G-fix}
Let $\X=\Spf A$ be an affine formal $R$-scheme and let $G$ be a finite group acting continuously on $A$.
Then
\[
(\Spf A)^G \cong \Spf A_G.
\]
\end{lemma}

\begin{proof}
If $L$ is an adic $R$-algebra, the quotient map $A\to A_G$ defines an injection between the sets of continuous $R$-homomorphisms
$
\Phi \colon \Hom_{R}(A_G,L)\hookrightarrow\Hom_R(A,L)
$.
A homomorphism $\sigma\colon A\to L$ belongs to the image of $\Phi$ if and only if $\sigma(a)=\sigma\big(g(a)\big)$ for every $a$ in $A$ and every $g$ in $G$, that is if and only if $\sigma$ belongs to $\big(\Hom_R(A,L)\big)^G$.
Therefore $\Phi$ induces a bijection $\Hom_{R}(A_G,L)\stackrel{\sim}{\to}\big(\Hom_R(A,L)\big)^G$, and since this bijection is functorial in $L$ this proves the lemma.
\end{proof}

We can now move to the study of $G$-fixed loci in the category of $K$-analytic spaces.
Let $\calA$ be a $K$-affinoid algebra and let $G$ act on $\calM(\calA)$, trivially on $K$.
Then $G$ acts on $\calA$, the ring of coinvariants $\calA_G$ is a $K$-affinoid algebra, and the same argument given in the case of special formal schemes proves that the $G$-fixed locus of $\calM(\calA)$ is the affinoid $K$-analytic space $\calM(\calA_G)$.
Now, if $\X=\Spf A$ is an affine special formal $R$-scheme and $G$ is a finite group acting on $\X$, with the trivial action on $R$, we obtain by functoriality an action of $G$ on $\X^\beth$, trivial on $K$.

\begin{remark}
This action can be constructed explicitly as follows.
Denote by $I$ the largest ideal of definition of $A$ and, as in the construction of $\X^\beth$ in Section \ref{section_semi-affinoids}, write $A_n=A\left[I^n/\pi\right]$, let $B_n$ be the $I$-adic completion of $A_n$, and set $C_n=B_n\otimes_RK$ for every $n>0$, so that $\X^\beth=\bigcup_n\calM(C_n)$.
If $g$ is an element of $G$, then $g$ acts on $A$ as a continuous ring morphism $g\colon A\to A$ which is invertible and whose inverse is continuous.
Therefore we have $g(I)=I$, since $g^{-1}(I)\subset I$ by continuity of $g$ and $g(I)\subset I$ by continuity of $g^{-1}$.
It follows that for every $n>0$ the element $g$ induces an action on $A_n$, hence actions on $B_n$ and $C_n$.
Since the $C_n\to C_{n+1}$ are equivariant, then $g$ induces the wanted action on $\X^\beth$. 
See Lemma~\ref{lemma_computation_G-fix_Weil_basechange} for an example of a computation of such an action.
\end{remark}

We can now describe the $G$-fixed locus of the semi-affinoid space $\X^\beth$.

\begin{proposition}\label{proposition_beth_commutes_G-fix}
Let $\X$ be an affine special formal $R$-scheme and let $G$ be a finite group acting on $\X$, with trivial action on $R$. 
Then we have
\[
\big( \X^\beth \big)^G \cong \big(\X^G\big)^\beth.
\]
\end{proposition}

\begin{proof}
Let $\rho\colon A \to A_G$ be the quotient map and set $I_G=\rho(I)$.
Then $A_G$ is a special $R$-algebra with ideal of definition $I_G$.
For any $n>0$, write $A'_n=A_G\left[(I_G)^n/\pi\right]$, let $B'_n$ be the $I_G$-adic completion of $A'_n$, and set $C'_n=B'_n\otimes_RK$, so that $(\Spf A_G)^\beth=\bigcup_n\calM(C'_n)$.
Since $\big(\X^\beth\big)^G=\bigcup_n\calM(C_n)^G=\bigcup_n\calM\big((C_n)_G\big)$, to prove the proposition it is now enough to show that for every $n>0$ the map $\rho$ induces an isomorphism $(C_n)_G\cong C'_n$.
In order to prove that, if we call $\alpha_n\colon A_n\to A'_n$ the natural surjection defined by sending an element of the form $a/\pi$, $a\in A$, to $\alpha(a/\pi)=\rho(a)/\pi$, it is sufficient to show that the kernel $\ker(\alpha_n)$ of $\alpha_n$ coincides with the ideal $J$ of $A_n$ generated by the elements of the form $x-g(x)$ for some $x$ in $A_n$ and some $\sigma$ in $G$.
If $a/\pi$ is an element of $\ker(\alpha_n)$ then $0=\alpha_n(a\pi)=\rho(a)/\pi$, so that $\rho(a)=0$ and therefore $a/\pi\in J$.
Conversely, consider an element of $A_n$ of the form $x=a/\pi-\sigma(a/\pi)=a/\pi-\sigma(a)/\pi$.
Then we have $\alpha_n(x)=\rho\big(a-\sigma(a)\big)/\pi=0$.
This shows that $J\subset \ker(\alpha_n)$, concluding the proof.
\end{proof}

\begin{remark}\label{remark_extension_action_canonical_model}
If $G$-acts on a reduced semi-affinoid space $X$, then this action extends canonically to its canonical model $\X=\Spf \big( \OOO_{X}(X)\big)$.
Indeed if $f\in\OOO_{X}(X)$ and $\sigma\in G$ then $\sigma$ acts on $f$ by sending it to $f\circ (\sigma^\beth)$, which is still bounded by $1$.
This action is continuous since it restricts to an action on the biggest ideal of definition of $\OOO_{X}(X)$, which is $\big\{f\in\OOO_{X}(X) \;\big|\; |f(x)|<1 \text{ for all }x\in X'\big\}$.
\end{remark}



\section{Descent of semi-affinoid spaces}
\label{section_descent_model}

Let $K'$ be a finite tamely ramified Galois extension of $K$ with Galois group $G$, let $X$ be a $K$-analytic space and denote by $X'=X\otimes_KK'$ its base change to $K'$.
In this section we explain how we can use Weil restrictions and $G$-fixed loci to describe a model of $X$ in terms of a model of $X'$, when the latter is a semi-affinoid $K'$-analytic space.
\medskip

Group actions naturally induce actions on Weil restrictions as follows.
Let $\calC$ be a category with fiber product.
If $G$ is a finite group acting equivariantly on morphisms $\X\to\s'$ and $\s'\to \s$ in the category $\calC$, with trivial action on $\s$, then we have a $G$-action on $\prod_{\s'|\s}\X$ defined by  
\(
\varphi \cdot \sigma=\sigma_\X\circ\varphi\circ\sigma_{\T'},
\)
where $\T\to\s$ is a morphism, $\T'=\T\times_\s\s'$, $\varphi\in \prod_{\s'|\s}\X(\T)$, $\sigma\in G$, $\sigma_{\X}$ is the automorphism induced by $\sigma$ on $\X$, and $\sigma_{\T'}=\sigma_{\s'}\times 1_\T$ where $\sigma_{\s'}$ is the automorphism induced by $\sigma$ on $\s'$.

Since dilated Weil restrictions are not defined by a universal property, the argument above does not apply to them.
Assume now that $G$ acts on an affine special formal $R'$-scheme $\X'$.
Then we obtain a $G$-action on the $K'$-analytic space $\big(\X'\big)^\beth$, hence a $G$-action on $\prod\nolimits_{K'|K}\big(\X'\big)^\beth$.
Whenever this action extends to a $G$-action on $\prod\nolimits_{R'|R}^\mathrm{dil}\X'$, which by Proposition~\ref{proposition_semi-affinoid_restriction} is a model of the latter, we say that the $G$-action on the dilated Weil restriction is \emph{compatible with the $G$-action} on $\big(\X'\big)^\beth$.
Such an action is described explicitly in the case of a base change in Lemma~\ref{lemma_computation_G-fix_Weil_basechange}, while explicit computations in the case of models of annuli are performed in Section~\ref{section_forms_annuli}.
We can now state the main result of the section.

\begin{theorem}\label{theorem_descent_model}
Let $K'|K$ be a finite tamely ramified Galois extension of discretely valued fields with Galois group $G$, let $R'|R$ be the corresponding extension of valuation rings, and let $X$ be a separated $K$-analytic space.
Assume that $X'=X\otimes_KK'$ is a semi-affinoid $K'$-analytic space, and let $\X'$ be an affine model of $X'$ such that $G$ acts on the dilated Weil restriction $\prod\nolimits_{R'|R}^\mathrm{dil}\X'$ compatibly with the natural semilinear Galois action on $X'$.
Then we have 
\[
\Big(\prod\nolimits_{K'|K}X'\Big)^G \cong \bigg(\Big(\prod\nolimits_{R'|R}^\mathrm{dil}\X'\Big)^G\bigg)^\beth\cong X.
\]
\end{theorem}

The proof of this result relies on an explicit computation which is the content of the following lemma.

\begin{lemma}\label{lemma_computation_G-fix_Weil_basechange}
Let $K'|K$ be a finite Galois extension of discretely valued fields with Galois group $G$, let $R'|R$ be the corresponding extension of valuation rings, and let $A$ be a special $R$-algebra. 
Moreover, assume that one of the following properties holds:
\begin{enumerate}
\item $K'$ is totally tamely ramified over $K$;
\item $K'$ is unramified over $K$.
\end{enumerate}
Then $G$ acts on $\prod\nolimits_{R'|R}^\mathrm{dil}\Spf ( A \otimes_R R' )$ compatibly with the natural Galois action on $\Spf ( A \otimes_R R' )^\beth$, and we have
\[
\Big( \prod\nolimits_{R'|R}^\mathrm{dil} \Spf ( A \otimes_R R' ) \Big)^G \cong \Spf A.
\]
\end{lemma}

\begin{proof}

\emph{Case (i)}. 
Write $m=[K':K]$ and let $\varpi$ be a uniformizer of $R'$ such that $\varpi^m\in R$. 
Then $\{1, \varpi, \ldots, \varpi^{m-1}\}$ is a basis of $R'$ as a free $R$-module. 
Moreover, $G$ is in this case cyclic, and given a generator $\sigma\in G$ the action of $G$ on $R'$ is determined by $\sigma(\varpi)=\zeta \varpi$, where $\zeta$ is a primitive $m$-th root of unity in  $R'$.
Write $A=R\{\underline X\}[[\underline Y]]/I$, where $\underline X$ and $\underline Y$ denote finite sets of variables.
Proposition~\ref{proposition_dilated_restriction_explicit} yields $\prod_{R'|R}^\mathrm{dil}\Spf (A\otimes_RR')=\Spf B$, with 
\[
\resizebox{1\textwidth}{!}{$
	B=\frac{R\big\{\underline X_0,|\varpi| \underline X_{1}, |\varpi|^2\underline X_{2}, \ldots, |\varpi|^{m-1} \underline X_{m-1}\big\}\big[\big[\underline Y_{0}, |\varpi| \underline Y_{1},\ldots,|\varpi|^{m-1}\underline Y_{m-1}\big]\big]}{I^{\mathfrak{c}}}.
$}
\]
In order to describe explicitly the action of $\sigma$ on $B$, let $T$ be one of the variables among $\underline X$ and $\underline Y$.
Since $\sigma$ acts trivially on the variables of $A\otimes_RR'$, by writing
\[
T=T_0+\varpi T_1+\ldots+\varpi^{m-1}T_{m-1}
\]
we have
\[
\sigma(T)=
\sigma(T_0)+\zeta\varpi\sigma(T_1)+\ldots+\zeta^{m-1}\varpi^{m-1}\sigma(T_{m-1})=
T,
\]
so that we can deduce that the action of $G$ on $B$ must verify $\sigma(T_i)=\zeta^i T_i$ for every $i\geq1$.
Since the ramification is tame, $\zeta^i-1$ is invertible in $R$ for every $1 \leq i \leq m-1$. 
Hence, for every such $i$, the variables $T_i$ vanish in $B_G$.
Since this is true for all the variables in $\underline X$ and $\underline Y$, we deduce that
$B_G=R\{\underline X_0\}[[\underline Y_0]]/(I^{\mathfrak c}\cap R\{\underline X_0\}[[\underline Y_0]])$.
From the definition of $I^\mathfrak c$ it readily follows that $B_G\cong A$, which by Lemma~\ref{lemma_representability_G-fix} is what we wanted to prove.

\emph{Case (ii)}. 
Let $k,k'$ be the residue fields of $R, R'$. 
Since $K'$ is Galois and unramified over $K$, it follows that $k'$ is Galois over $k$, with $[k':k]=[K':K]=m$, and we have a natural isomorphism $\Gal(K'|K)\cong \Gal(k'|k)$. 
Write $\Gal(K'|K)=\{\sigma_0=\id, \dots, \sigma_{m-1}\}$.
By the normal basis theorem (see for example \cite[VI, \S13, Theorem 13.1]{Lang}) there exists a basis $(\overline{a_0}, \ldots, \overline{a_{m-1}})$ of $k'$ over $k$ such that $\overline{\sigma_i}(\overline{a_0})=\overline{a_i}$ for every $i$.
As a result, given an element $a_0\in R'$ lifting $\overline{a_0}$, the set $\{\sigma_i(a_0): 0\leq i \leq m-1 \}$ is a basis of $R'$ as a free $R$-module.
As before, write $A=R\{\underline X\}[[\underline Y]]/I$, where $\underline X$ and $\underline Y$ denote finite sets of variables.
Since $K'|K$ is unramified the dilated Weil 
restriction coincides with the classical one, and so by Lemma~\ref{lemma_restriction_affine} we have $\prod_{R'|R}^\mathrm{dil}\Spf (A\otimes_RR')=\Spf B$, where
\pagebreak
\[
B=\frac{R\{\underline X_0, \ldots, \underline X_{m-1}\}[[\underline Y_{0}, \ldots, \underline Y_{m-1} ]]}{I^{\mathfrak{c}}}.
\]
Let $T$ be one of the variables among $\underline X$ and $\underline Y$, and write
\[
T=a_0T_0+\sigma_1(a_0)T_1+\ldots+\sigma_{m-1}(a_0)T_{m-1}.
\]
Observe that each $\sigma_i$ permutes the elements of the set $\{\sigma_i(a_0): 0\leq i \leq m-1 \}$, because $\sigma_i$ is an automorphism and for each $j$ we have $\sigma_i\circ\sigma_j=\sigma_{k_{ij}}$ for some $0\leq k_{ij} \leq m-1$.
Therefore, since $\sigma$ acts trivially on $T\in A\otimes_RR'$, the $G$-action on $B$ must verify $\sigma_i(T_j)=T_{k_{ij}}$ for every $i$ and $j$.
Moreover, since for every $j$ the association $i\mapsto k_{ij}$ is bijective, for every $j$ and $k$ there exists $i$ such that $\sigma_i(T_j)=T_k$.
It follows that $B_G=B/(\underline X_i-\underline X_j,\underline Y_i-\underline Y_j)_{i,j}$.
As in case (i), we deduce that
$B_G=R\{\underline X_0\}[[\underline Y_0]]/(I^{\mathfrak c}\cap R\{\underline X_0\}[[\underline Y_0]])$.
This shows that $B_G\cong A$ which, together with Lemma~\ref{lemma_representability_G-fix}, concludes the proof.
\end{proof}

\begin{proof}[Proof of Theorem~\ref{theorem_descent_model}]
We begin by showing that the statement of the theorem is local on $X$. 
Let $\{X_i\}$ be a cover of $X$ by affinoid subsets and set $X_i'=X_i\otimes_KK'$ for every $i$, so that we have $X'= \cup X_i'$.
The Weil restriction functor gives immersions $\prod\nolimits_{K'|K}X_i' \hookrightarrow \prod\nolimits_{K'|K}X'$ for every $i$, and $\prod\nolimits_{K'|K}X'$ is obtained by gluing $\prod\nolimits_{K'|K}X_i'$ along the natural isomorphisms induced by $\prod\nolimits_{K'|K}X_i' \cap \prod\nolimits_{K'|K}X_j' \cong \prod\nolimits_{K'|K} \big (X_i' \cap X_j' \big)$.
Finally, since the action of $G$ on $X'$ arises from base-change, the isomorphisms involved in the gluing procedure are $G$-equivariant by construction. 
It follows that $\big(\prod\nolimits_{K'|K}X'\big)^G$ is the gluing of the $\big(\prod\nolimits_{K'|K}X_i'\big)^G$.
Hence, if the theorem holds for every $X_i$, we have that $\big(\prod\nolimits_{K'|K}X'\big)^G$ is obtained by gluing the $X_i$ along their intersections, so that $\big(\prod\nolimits_{K'|K}X'\big)^G= X$.
Let us now show that the theorem holds for $X$ affinoid $K$-analytic space. If $K''$ is an intermediate extension between $K$ and $K'$, by the universal property of the representing object we have that $\prod\nolimits_{K''|K} \big(\prod\nolimits_{K'|K''}X' \big) = \prod\nolimits_{K'|K}X'$. 
For the same reason, if the extensions $K'|K''$ and $K''|K$ are Galois, with respective Galois groups $G_1$ and $G_2$, then we have
\[
\bigg ( \prod\nolimits_{K''|K} \Big(\prod\nolimits_{K'|K''}X' \Big)^{G_1} \bigg)^{G_2} = \Big( \prod\nolimits_{K'|K}X'\Big)^G.
\]
Now, since the extension $K'|K$ is the composition of the two Galois extensions $K'|K^{\mathrm{ur}}$ and $K^{\mathrm{ur}}|K$, where $K^{\mathrm{ur}}$ is the maximal unramified extension of $K$ contained in $K'$, and therefore $K'$ is totally ramified over $K^{\mathrm{ur}}$, we can assume without loss of generality that $K'|K$ satisfies the hypothesis of Lemma~\ref{lemma_computation_G-fix_Weil_basechange}.
To conclude, observe that if $\Spf A$ is a model of $X$ then we have $\big( \Spf  (A \otimes_R R') \big)^\beth\cong X' \cong (\X')^\beth$.
Combining this fact with Propositions \ref{proposition_semi-affinoid_restriction} and \ref{proposition_beth_commutes_G-fix}, we obtain
\begin{align*}
\bigg(\Big(\prod\nolimits_{R'|R}^\mathrm{dil}\X'\Big)^G\bigg)^\beth & \overset{\ref{proposition_beth_commutes_G-fix}}{\cong} \bigg(\Big(\prod\nolimits_{R'|R}^\mathrm{dil}\X'\Big)^\beth\bigg)^G \overset{\ref{proposition_semi-affinoid_restriction}}{\cong} \Big(\prod\nolimits_{K'|K}\big(\X'\big)^\beth\Big)^G 
\\
& \cong \Big(\prod\nolimits_{K'|K} \big( \Spf  (A \otimes_R R') \big)^\beth \Big)^G 
\\
& \overset{\ref{proposition_semi-affinoid_restriction}}{\cong} \bigg( \Big( \prod\nolimits_{R'|R}^\mathrm{dil} \Spf ( A \otimes_R R' ) \Big)^\beth\bigg)^G \\
& \overset{\ref{proposition_beth_commutes_G-fix}}{\cong} \bigg( \Big( \prod\nolimits_{R'|R}^\mathrm{dil} \Spf ( A \otimes_R R' ) \Big)^G\bigg)^\beth.
\end{align*}
The theorem now follows by applying Lemma \ref{lemma_computation_G-fix_Weil_basechange}.
\end{proof}

The Weil restriction and the G-fix locus of a semi-affinoid analytic space under finite tamely ramified extensions are both semi-affinoid, so we obtain the following corollary of Theorem \ref{theorem_descent_model}.
\begin{corollary}
Let $K'$ be a finite, tamely ramified extension of $K$, and let $X$ be a $K$-analytic space such that $X'=X\otimes_KK'$ is a semi-affinoid $K'$-analytic space. Then $X$ is semi-affinoid.
\end{corollary}

\begin{remark}
From the computation of $\big( \prod\nolimits_{R'|R}^\mathrm{dil} \Spf ( A \otimes_R R' ) \big)^G$ in the proof of Lemma \ref{lemma_computation_G-fix_Weil_basechange}, one can see that this result remains true if we replace the dilated restriction with the usual Weil restriction. 
However, this is not sufficient to prove Theorem~\ref{theorem_descent_model}, as the isomorphisms provided by Proposition \ref{proposition_semi-affinoid_restriction} no longer hold when $K'$ is ramified over $K$. 
\end{remark}

\begin{remark}
Following the notation of the theorem, we obtain a model of $X$ if we consider the flatification of the dilated Weil restriction $\X=\prod\nolimits_{R'|R}^\mathrm{dil}\X'$, which is the affine special formal $R$-scheme $\Spf\big(\OO_\X(\X)/\pi\text{-torsion}\big)$. 
Moreover, the theorem holds if we replace $\X$ by its integral closure in its generic fiber. 
Since $K'$ is separable over $K$, $X'$ is reduced if and only if $X$ is reduced. 
When this is the case, and $\X'=\Spf \big( \OOO_{X'}(X')\big)$ is the canonical model of $X'$, the $G$-action on $X'$ extends canonically to $\X'$, as observed in Remark~\ref{remark_extension_action_canonical_model}.
Putting all of this together, we obtain the canonical model of $X$ by taking the integral closure in its generic fiber of the flatification of the formal scheme produced by Theorem~\ref{theorem_descent_model}. 
\end{remark}
%


\section{Annuli and their automorphisms} \label{section_annuli}

In this section we study automorphisms of finite order of annuli, not only of closed annuli but also of open and semi-open ones.
Using techniques reminiscent of the theory of Newton polygons, we prove an analogue of Weierstrass preparation theorem for bounded functions on annuli.
We use this result to show that, for a suitable choice of a presentation, every tame automorphism of finite order of an annulus is linear.
\medskip

A $K$-analytic space $V$ is said to be an \emph{open annulus} (or \emph{semi-open annulus}, or \emph{closed annulus}) if it is a semi-affinoid $K$-analytic space having a model of the form $\Spf A$, with $A \cong R[[X,Y]]/(XY-\pi^e)$ (respectively $A\cong R\{X\}[[Y]]/(XY-\pi^e)$, or $A \cong R\{X,Y\}/(XY-\pi^e)$) for some positive integer $e\in \N$.
Given such a presentation, the pair $(X,Y)$ of elements of $A$ is referred to as a \emph{Laurent pair} for the annulus $V$, and the integer $e$ as the \emph{modulus} of the annulus.

Observe that $V$ is distinguished, $\Spf A$ is its canonical model, and by Lemma~\ref{lemma_canonical_model} we have $A\cong \OOO_V(V)$. 
The canonical reduction of an open (or semi-open, or closed) annulus is $\Spf A_k$, where $A_k\cong k[[X,Y]]/(XY)$ (respectively $A_k\cong k[X][[Y]]/(XY)$, or $A_k\cong k[X,Y]/(XY)$); in all three cases it has two irreducible components, of generic points $(X)$ and $(Y)$. 
As usual, if $f\in A$ is a bounded function on $V$, we denote by $\overline f$ its reduction, that is its image in the ring $A_k$.

\begin{remark}\label{units_in_annuli}
If $A$ is a special $R$-algebra, an element $f$ of $A$ is a unit if and only if its image $\overline f$ in $A_k$ is a unit.
In particular, a bounded function $f(X,Y)=\sum_{i\geq0} a_i X^i + \sum_{i>0} b_i Y^i$ on an open (or semi-open, or closed) annulus is invertible if and only if $\ord_\pi(a_0)=0$ (respectively $\ord_\pi(a_0)=0$ and $\ord_\pi(a_i)>0$ in the semi-open case, or $\ord_\pi(a_0)=0$, $\ord_\pi(a_i)>0$, and $\ord_\pi(b_i)>0$ in the closed case).
\end{remark}

\begin{remark}\label{remark_modulus_intrinsic}
Observe that two annuli which are isomorphic over $K$ have the same modulus.
Indeed, if $V$ is a $K$-annulus of modulus $e$, then $e+1$ is the number of irreducible components of the special fiber of the minimal regular model of $V$ over $R$.
An interpretation in terms of the geometry of the associated Berkovich space is the following.
If $V$ is a $K$-annulus, then the topological space underlying the associated Berkovich space is an infinite tree that retracts by deformation onto its skeleton, which is the unique subset that connects the two points of the boundary of $V$ and is homeomorphic to an interval.
If $V$ has modulus $e$, then its skeleton contains in its interior exactly $e-1$ points onto which a $K$-rational point retracts.
\end{remark}

To study analytic functions on annuli it is useful to consider some valuations of the ring $\OOO_V(V)$.
For the reader's convenience, we first recall what happens for analytic functions on discs.
Let $f=\sum_{i\geq0}a_i X^i \in R[[X]]$ be a bounded function on the open unit $K$-analytic disc $D^-=\big\{x\in\A^{1,\mathrm{an}}_K\,\big|\,|X(x)|<1\big\}$, and denote by $\eta(f)=\min_{i}\{\ord_\pi(a_i)\}$ the valuation of $f$ at the Gauss point of the disc.
Then the smallest degree of a monomial of the reduction $\overline{f/\eta(f)}\in k[[X]]$ of $f/\eta(f)$ is $v(f)=\min \big\{i \in \Z    \,\,\big|\,\, \ord_\pi(a_i) = \eta_X(f) \big\}$.
It follows from the Weiestrass preparation theorem that $f$ has $v(f)$ zeros on $D^-$.
If moreover $f$ converges on the closed unit $K$-analytic disc $D$, that is it belongs to $R\{X\}$, then $\overline{f/\eta(f)}$ is a polynomial of degree $\nu(f)=\max \big\{i \in \Z    \,\,\big|\,\, \ord_\pi(a_i) = \eta_X(f) \big\}$, and so $f$ has $\nu(f)$ zeros on $D$.

Assume now that $V$ is a $K$-analytic annulus, and let $(X,Y)$ be a Laurent pair for $V$.
We define \emph{boundary valuations} associated with each irreducible component of the canonical reduction of $V$.
To the component corresponding to the ideal $(X)$ one attaches the rank 2 valuation on $\OOO_V(V)$ defined by $f\mapsto \big(\eta_X(f), v_X(f)\big)\in \Z_{\geq0}\times\Z$, where we write $f=\sum\nolimits_{i \in \Z} a_i X^i$ as a function of $X$ and set
\begin{align*}
\eta_X \Big(\sum\nolimits_{i \in \Z} a_i X^i \Big)=&\,\min_{i}\{\ord_\pi(a_i)\},
 \\ 
v_X\Big(\sum\nolimits_{i \in \Z} a_i X^i\Big)=&\, \min \big\{i \in \Z    \,\,\big|\,\, \ord_\pi(a_i) = \eta_X(f) \big\}. 
\end{align*} 
In the same way, by expressing $f$ as a function of $Y$, one can associate with the component corresponding to the ideal $(Y)$ a rank 2 valuation $f\mapsto \big(\eta_Y(f), v_Y(f)\big)$.

\begin{remark}
The (unordered) pair of boundary valuations of an annulus does not depend on the choice of a Laurent pair.
Indeed, if $(X',Y')$ is another Laurent pair, then there exists a unit $u$ of $\calO^\circ_V(V)$ such that either $X'=uX$ or $X'=uY$, and the boundary valuations vanish on $u$.
This independence is reflected by the fact that boundary valuations are the two type 5 points pointing toward the interior in the boundary of the adic space of the annulus.
\end{remark}

Finally, if $V$ is a closed annulus, with every element $f=\sum\nolimits_{i \in \Z} a_i X^i$ of $\OOO_V(V)$ we associate the integer
\begin{align*} 
\nu_X(f)=\max\big\{i \in \Z  \,\,\big|\,\, \ord_{\pi}(a_i)=\eta_X(f)\big\}.
\end{align*}

We are going to give a geometric interpretation of the valuations introduced above, and deduce a Weierstrass preparation result for analytic functions on annuli.
The basic step is the following lemma.
To simplify the notation, we write $f(a)=f(a,a^{-1}\pi^e)$.

\begin{lemma}\label{lemma_Weierstrass_division}
Let $f(X,Y)\in R\{X,Y\}/(XY-\pi^e)$ be a bounded function on a closed $K$-analytic annulus. Then:
\begin{enumerate}
\item If there exists an element $a$ of $R^\times$ such that $f(a)=0$, we can write 
\[
Y f(X,Y) = (Y-a^{-1}\pi^e)g(X,Y)
\]
for some function $g(X,Y)$ in $R\{X,Y\}/(XY-\pi^e)$.
\item The function $f$ has $\nu_X(f)- v_X(f)$ many zeros on the subset of $V$ defined by $|X|=1$.
\end{enumerate}
\end{lemma}
\begin{proof}
Write $f(X,Y)=\sum_{i\geq 0} a_i X^i + \sum_{i> 0} b_iY^i$, and assume that $f(a)=0$.
By replacing $f(X,Y)$ with $f(aX, a^{-1}Y)$, we can assume without loss of generality that $a=1$, so that $f(1)=\sum a_i + \sum b_i \pi^{ei}=0$.
Set $c_i= - \sum_{k\geq i+1} a_k$ and $d_i=\sum_{k\geq i}b_k\pi^{e(k-i)}$.
Then $g(X,Y)= \sum_{i\geq 0} c_i X^i + \sum_{i> 0} d_iY^i$ is an element of $R\{X,Y\}/(XY-\pi^e)$.
Comparing coefficients, one verifies that $Y f(X,Y) = (Y-\pi^e)g(X,Y)$, which proves $(i)$.
Now, for a general $f$, observe that $\nu_X(f)=v_X(f)$ if and only if the reduction $\overline{f/\eta_X(f)}$ of $f/\eta_X(f)$ is a monomial of degree ${v_X(f)}$, that is if and only if all zeros of $f$ are contained in the subset of $V$ on which $|X|<1$.
If $f(a)=0$, without loss of generality, after replacing $K$ by an extension, we can assume that $a$ belongs to $R$.
By the previous part we can then write $Y f(X,Y)=(Y-a^{-1}\pi^e)g(X,Y)$, and since $\nu_X(Y)=v_X(Y)=-1$ and $\nu_X(Y-a^{-1}\pi^e)=1=v_X(Y-a^{-1}\pi^e)+1$, it follows that $\nu_X(f)- v_X(f)=\nu_X(g)- v_X(g)+1$.
We can repeat this argument as long as $\nu_X(g)- v_X(g)>0$, until we obtain a factorization of the form $Y^{\nu_X(f)-v_X(f)}\cdot f(X,Y)=P(Y)\cdot g(X,Y)$, with $P$ a polynomial having $\nu_X(g)- v_X(g)>0$ zeros in $W=\big\{x\in K^\mathrm{alg} \,\big|
\, |x|=1\big\}$ and satisfying $\nu_X(P)-v_X(P)=\nu_X(f)- v_X(f)$, and $g$ nowhere-vanishing on $W$.
This proves $(ii)$.
\end{proof}

Let $f$ be a bounded function on $V$, which as before we write as $f(X)=\sum_{i\in\Z} a_i X^i$.
Given a real number $0\leq r\leq e$, consider the positive real number
\[
\eta_r(f)= \min_{i\in\Z}\big\{\ord_\pi(a_i)+ir\big\}.
\]
Observe that $\eta_{0}(f)=\eta_X(f)$ and $\eta_{e}(f)=\eta_Y(f)$, and that the map $r \mapsto \eta_r(f)$ defines a piecewise linear function on $[0, e]$.
Moreover, the derivative of this function is discontinuous at a point $r\in]0,e[$ if and only if the value defining $\eta_r(f)$ is attained by more than one monomial of $f$.
This can happen only for finitely many values of $r$, since the valuation of $R$ is discrete, and those values are all rational.
We say that a rational number $0\leq r\leq e$ (with $0\leq r<e$ if $V$ is a semi-open annulus, and $0<r<e$ if $V$ is an open annulus) is a \emph{critical radius} of $f$ if the value defining $\eta_r(f)$ is attained by more than one monomial of $f$.

Observe that, if $r$ is rational, then after passing to the extension $K(\pi^r)$ of $K$ we have $\eta_r\big(f(X)\big)=\eta_X\big(f(\pi^rX)\big)$.
In particular, it follows from the definition that $r$ is a critical radius of $f$ if and only if $\nu_X \big(f(\pi^rX) \big) \neq v_X \big(f(\pi^rX)\big)$, or equivalently, by part $(ii)$ of Lemma~\ref{lemma_Weierstrass_division}, whenever $f(a)=0$ for some element $a$ of $K^\mathrm{alg}$ such that $|a|=|\pi|^r$.
In particular, since $f$ has finitely many critical radii, by Lemma~\ref{lemma_Weierstrass_division} it can only have finitely many zeros on $V$.

More generally, given a rational number $0<r_0<1$, the right derivative of $\eta_r(f)$ at $r_0$ is $v_X\big(f(\pi^rX)\big)$, while its left derivative at $r_0$ is $\nu_X\big(f(\pi^rX)\big)$, since in a small neighborhood of $r_0$ we have $\eta_r(f)=\min\big\{\ord_\pi(a_i)+ir \big|i\in\{v_X(f(\pi^rX)),\nu_X(f(\pi^rX))\}\big\}$.
Combined with Lemma~\ref{lemma_Weierstrass_division}, this makes it possible to count the zeros of $f$ on any open or closed sub-annulus of $V$ in terms of $v_X$ and $\nu_X$.

It is now simple to prove the following form of Weierstrass preparation for annuli. 

\begin{proposition}[Weierstrass preparation]\label{proposition_weierstrass_preparation}
Let $V$ be a $K$-analytic annulus, let $(X,Y)$ be a Laurent pair for $V$, and let $f\in \OOO_V(V)$ be a bounded analytic function on $V$.
Then there exist a monic polynomial $P\in R[Y]$ and a unit $u$ of $\OOO_V(V)$ such that 
\[
Y^{\alpha}f(X,Y)=\pi^{\eta_Y(f)}P(Y)u(X,Y),
\]
where $\alpha=v_X(f)$ if $V$ is an open annulus and $\alpha=\nu_X(f)$ otherwise.
\end{proposition}
\begin{proof}
We begin by claiming that we can write 
$$Y^{\alpha-v_X(g)} f(X,Y)=P(Y) g(X,Y),$$
where $P(Y)$ belongs to $R[Y]$ and the function $g\in \OOO_V(V)$ has no zeros on $V$.
Indeed, we can pass to an extension $K'$ of $K$ such that the zeros $a_i,\ldots,a_n$ of $f$ on $V$ are all $K'$-rational, so that by applying Lemma~\ref{lemma_Weierstrass_division} $n$ times we obtain $Y^n f(X,Y)=P(Y) g(X,Y)$, with $P(Y)=\prod_{i=1}^n(Y-\pi^ea_i^{-1})$, and we have $n=\alpha-v_X(g)$ because $g$ has no zeros on $V$.
Since $f$ is defined over $R$, the $a_i$ are permuted by the Galois group of $K'$ over $K$, so that $P(Y)$ has coefficients in $R$, and since $P(Y)$ is monic then $g$ has coefficients in $R$ as well.
Now, since the function $g$ does not vanish on $V$, the value of $\eta_r(g)$ is attained for every $r$ by a unique monomial $a_{v_X(g)}X^{v_X(g)}$ of $g(X)$, with $\ord_\pi(a_{v_X})=\eta_X(g)$.
Hence, one has $g(X,Y)=\pi^{\eta_X(g)} X^{v_X(g)} u(X,Y)$, where $u$ is a unit of $\OOO_V(V)$.
Substituting this in the first relation and multiplying both sides by $Y^{v_X(g)}$, we obtain the equality $Y^{\alpha} f(X,Y)=\pi^{\eta_X(g)+ e\cdot v_X(g)} p(Y) u(X,Y)$.
We have $\eta_X(g)+ e\cdot v_X(g)=\eta_Y(g)$ because the dominant monomial of $g(X)$, which is $a_{v_X(g)}X^{v_X(g)}$, has smallest valuation also among the monomials in $Y$ when writing $g$ as $g(\pi^eY^{-1},Y)$ (to see this, observe that the constant term of $g/(a_{v_X(g)}X^{v_X(g)})$ has valuation zero, and this remains true when writing $g$ in $Y$), where it is written as $a_{v_X(g)}\pi^{e\cdot v_X(g)}Y^{-v_X(g)}$.
Moreover, since $P(Y)$ is monic we have $\eta_Y\big(P(Y)\big)=\eta_Y(Y^n)=0$, and so $\eta_Y(g)=\eta_Y(f)$, which concludes the proof.
\end{proof}

\begin{remark}
In the case of open annuli, a Weierstrass preparation theorem was proven by Henrio in \cite[Lemme 1.6.]{Henrio01}.
His proof relies on an explicit computation to reduce the proof to the classical Weierstrass preparation for open discs.
While it is possible to prove Proposition~\ref{proposition_weierstrass_preparation} extending Henrio's techniques, we believe that the method we employ allows a deeper understanding of the geometric nature of the result.
Note that if we assumed Proposition~\ref{proposition_weierstrass_preparation} in the case of open annuli, that is Henrio's result, we could apply it to the restriction of a bounded function $f(X,Y)$ on a closed (or semi-open) annulus to the biggest open annulus it contains. 
We would obtain an invertible function $u(X,Y)$ in ${R[[X,Y]]}/{(XY-\pi^e)}$, but $u$ might not be invertible, nor convergent, on the closed (or semi-open) annulus.
\end{remark}

In the remaining part of the section we study the automorphisms of finite order $m$ of a $K$-annulus $V$, when $m$ is coprime with the residual characteristic of $K$.
We do not suppose that such an automorphism $\sigma$ acts trivially on $K$.
However, being a morphism of analytic spaces, $\sigma$ respects the absolute value of $K$ and therefore induces an automorphism $\bar\sigma$ of the residue field $k$ of $K$. 
The canonical reduction $\Spec \big(\calO^\circ_V(V)_k\big)$ of $V$ has two irreducible components, with generic points $(X)$ and $(Y)$, where $(X,Y)$ is a Laurent pair for $V$, and so the action induced by $\sigma$ on it can either fix both points $(X)$ and $(Y)$, or exchange them. 
In the first case we say that $\sigma$ \emph{fixes the branches} of $V$; in the second case we say that $\sigma$ \emph{switches the branches} of $V$.
Observe that this notion does not depend on the choice of the Laurent pair $(X,Y)$.

\begin{remark}\label{remark_switched_skeleton}
The automorphism $\sigma$ fixes the branches of the annulus $V$ if and only if its action on the Berkovich space of $V$ fixes its skeleton pointwise; while when it switches the branches the action on the skeleton reflects it over a point.
This explains our terminology.
\end{remark}

We now prove a linearization result for the action on $\calO^\circ_V(V)$ of an automorphism that fixes the branches of $V$.
In the case of open annuli, a slightly weaker form of this result appears in Henrio's unpublished doctoral dissertation \cite[Chapitre II, Propositions 3.1 and 3.3]{Henrio1999}.

\begin{proposition}\label{proposition_linearization}
Let $V$ be a $K$-analytic annulus of modulus $e$, and let $\sigma$ be an automorphism of $V$ of finite order $m$ that fixes the branches of $V$ and acts on $\pi$ as multiplication by a (not necessarily primitive) $m$-th root of unity $\zeta$. Assume that the characteristic of the residue field $k$ of $K$ does not divide $m$. Then, there exist a Laurent pair $(X,Y)$ for $V$ and a unit $u$ of $R$ such that
\begin{align*}
\sigma(X) = &\, u X, \\ 
\sigma(Y) = &\, \zeta^e u^{-1} Y.
\end{align*}
Moreover, if $\zeta$ is primitive, then $u$ can be taken to be a $m$-th root of unity, whereas if $\zeta=1$, then one can take $u=1$.
\end{proposition}

\begin{proof}
Let $(X,Y)$ be a Laurent pair for $V$.
Since $\big(\sigma(X), \zeta^{-e}\sigma(Y)\big)$ is again a Laurent pair and $\sigma$ fixes the branches, one gets $\eta_X\big(\sigma(X)\big)=0, v_X\big(\sigma(X)\big)=\nu_X\big(\sigma(X)\big)=1$ and $\eta_Y\big(\sigma(X)\big)=e$, and analogous relations hold for $\sigma(Y)$.
By Proposition \ref{proposition_weierstrass_preparation}, we have that 
\[
Y\sigma(X)=\pi^eP(Y)U_1 \;\;\;\mbox{ and }\;\;\; X\sigma(Y)=\pi^eQ(X)U_2,
\]
with $U_1$ and $U_2$ units of ${\OOO_V(V)}$, and $P$ and $Q$ monic polynomials.
By multiplying these two relations one gets $\zeta^e\pi^{2e}=\pi^{2e}P(Y)Q(X)U_1U_2$, so that $P(Y)Q(X)$ is invertible in $\OOO_V(V)$, yielding $P=Q=1$ (because $P$ and $Q$ are monic).
It follows that $\sigma(X) = XU_1$.
Denote by $a\in R^\times$ the constant term of $U_1$.
Since $\sigma$ has order $m$, we have that $X={\sigma}^m(X)={U_1}\sigma({U_1})\cdots\sigma^{m-1}({U_1}) X$, so that $a\sigma(a)\cdots\sigma^{m-1}(a)\equiv 1$ modulo $(\pi,X,Y)$.
Now let $u$ be the multiplicative representative in $R$ of the image of $a$ in the residue field $k$ of $K$.
Since the choice of the multiplicative representatives commutes with automorphisms, we have $u\sigma(u)\cdots\sigma^{m-1}(u)= 1$ in $R$.
Setting
\[
\resizebox{1\textwidth}{!}{$
X'
= 
X
+
u^{-1}\sigma(X)
+
u^{-1}\sigma(u^{-1})\sigma^2(X)
+
\dots
+
u^{-1}\sigma(u^{-1})\cdots \sigma^{m-2}(u^{-1}) \sigma^{m-1}(X)
$}
\] 
we have that $\sigma(X')=u X'$. 
Moreover $X'$ is divisible by $X$ because every term of the sum is, and ${X'}/{X}\equiv m$ modulo $(\pi,X,Y)$ because every term reduces to 1.
Since the characteristic of $k$ does not divide $m$, this implies that $X'=\eta X$ for some invertible element $\eta$ of $\OOO_V(V)$. 
Then $Y'=\eta^{-1}Y$ is such that $(X',Y')$ is a Laurent pair for $V$, and we have $\sigma(X')\sigma(Y') = \sigma(\pi^e)= \zeta^e\pi^e$, so that $\sigma(Y')=\zeta^e u^{-1} Y'$.
To prove the final claim, observe that if $\zeta$ is primitive, then $\sigma$ is trivial on $k$.
In this case $a$ reduces to a $m$-th root of unity of $k$ modulo $(\pi,X,Y)$, and therefore $u$ is the unique lift to a $m$-th root of unity of $R$.
On the other hand, if $\zeta=1$, then $k|k^{<\overline{\sigma}>}$ is cyclic of degree $m$, and so the norm of $\bar{u}$ relative to the extension $k|k^{<\overline{\sigma}>}$ is unitary.
It then follows from Hilbert's Theorem 90 (see for example \cite[VI, \S6, Theorem 6.1]{Lang}) that there exists $b\in k^\times$ such that $b=\overline{\sigma(b)}\bar{u}$. 
Taking a multiplicative representative $v$ of $b$ in $R^\times$, we find that $v=u\sigma(v)$.
Hence, if we set $X''=vX'$, we have $\sigma(X'')=\sigma(v)uX'=vX'=X''$.
After setting $Y''=v^{-1} Y'$, we obtain a Laurent pair $(X'', Y'')$ such that the action of $\sigma$ is the identity on the coordinates.
\end{proof}

We conclude the section with a similar linearization result in the case of involutions that switch the branches of $V$.

\begin{proposition}\label{proposition_linearization_switched_branches}
Let $V$ be a $K$-analytic annulus of modulus $e$, let $\sigma$ be an automorphism of $V$ of order $2$ that acts on $\pi$ as the multiplication by an element $\zeta$ of $\{\pm 1\}$, and assume that $\sigma$ switches the branches of $V$.
Then $V$ is either an open or a closed annulus, $\zeta^e=1$, and there exist a unit $u$ of $R$ that is invariant under $\sigma$ and elements $X,Y$ of $\calO^\circ_V(V)$ such that 
\begin{align*}
\sigma(X) = &\, Y,\\ 
\sigma(Y) = &\, X,
\end{align*}
together with an isomorphism
\[
\calO^\circ_V(V) \cong \frac{R[[X,Y]]}{(XY - u \pi^e)}
\]
whenever $V$ is an open annulus, and 
\[
\calO^\circ_V(V) \cong \frac{R\{X,Y\}}{(XY - u \pi^e)}
\]
whenever $V$ is a closed annulus.
\end{proposition}
\begin{proof}
Let $(X,Y)$ be a Laurent pair for $V$.
First of all observe that $V$ cannot be semi-open, because the generic points of the two irreducible components of the canonical reduction of a semi-open annulus, which is the formal spectrum of $k[X][[Y]]/(XY)$, cannot be exchanged by an automorphism.
The following proof applies both to the open and to the closed case.
Using Proposition \ref{proposition_weierstrass_preparation}, one can write $\sigma(X)=U_1 Y$ and $\sigma(Y)=U_2 X$ for some units $U_1$ and $U_2$ of ${\OOO_V(V)}$.
We have $X= \sigma\big(\sigma(X)\big) = \sigma(U_1) U_2 X$, so that $\sigma(U_1) U_2 = 1$.
Morever, $\zeta^e \pi^e = \sigma(XY)=\sigma(X)\sigma(Y)=U_1 U_2 \pi^e$, then yielding $U_1 U_2=\zeta^e$.
Therefore, when $\zeta=-1$ the integer $e$ must be even.
Indeed, in this case $\sigma$ fixes $k$ and so, modulo $(\pi, X, Y)$, we have $1\equiv \sigma(U_1)U_2 \equiv U_1U_2 \equiv \zeta^e$. 
In particular, $U_1U_2=1$ regardless of what $\zeta$ is.
Up to performing a change of variables and replacing the condition that $(X,Y)$ is a Laurent pair with the weaker assumption that $XY=\xi \pi^e$ for some unit $\xi \in R^\times$, we can suppose that $U_1 \equiv U_2 \equiv 1$ mod $\pi$.
Indeed, if we call $x$, $y$, and $u_1$ the reductions modulo $\pi$ of $X$, $Y$, and $U_1$ respectively, we can replace $x$ with $u_1^{-1}x$ in order to get $\overline{\sigma}(x)=y$ and $\overline{\sigma}(y)=x$.
An easy computation (see \cite[Lemme 2.3.]{Henrio1999}) shows the existence of lifts $X$ and $Y$ of $x$ and $y$ to $\calO^\circ_V(V)$, and of a lift $\xi$ of $\overline{u_1}^{-1}$ to $R^\times$ such that $XY= \xi \pi^e$.
Hence, recalling that $U_1U_2=1$, this change of variables allows us to write $U_1= (1+ \pi U)^{-1}$ and $U_2 = (1+ \pi U)$ for some $U$ in ${\OOO_V(V)}$.
It follows that $\sigma(\pi U)= \pi U$.\\
Our goal is to find a unit $\eta$ of ${\OOO_V(V)}$ such that $\eta \sigma(\eta) = vU_1^{-1}$ for some unit $v$ of $R$ satisfying $\sigma(v)=v$. 
Indeed, once we have done so, we can set $X'=v^{-1} \eta X$, and $Y'=\eta^{-1}Y$, so that these new variables satisfy 
\[
\sigma(X')= v^{-1} \sigma(\eta) U_1 Y= \eta^{-1} Y = Y'
\] 
and 
\[
\sigma(Y')=\sigma(\eta)^{-1} U_1^{-1} X = v^{-1}\eta X = X'.
\]
Moreover, we have $X'Y'=v^{-1}\eta X \eta^{-1} Y= v^{-1} XY = v^{-1} \xi \pi^e=u\pi^e$, where $u=v^{-1} \xi$, and since $\sigma(u)=u$ this would conclude the proof of the theorem.
\\
Therefore, it remains to find such $\eta$ and $v$.
In order to do so, we proceed by successive approximations, constructing sequences of functions $(\eta_n)_{n\in\N}$ in ${\OOO_V(V)}^\times$ and elements $(v_n)_{n\in\N}$ of $R^\times$ such that
\[
\eta_n\sigma(\eta_n)-v_n(1+\pi U) \equiv 0 \; \; \mod \pi^{n+1}
\]
for every $n\in\N$.
We first set $\eta_0=v_0=1$, and suppose that we defined successfully $v_n$ and $\eta_n$.
We then write $\eta_n\sigma(\eta_n) - v_n(1+\pi U) = \pi^{n+1}g(X,Y)$ in $\OOO_V(V)$.
Since $\sigma\big(\eta_n\sigma(\eta_n)\big)=\eta_n\sigma(\eta_n)$, $\sigma(v_n)=v_n$, and $\sigma(\pi U)= \pi U$, then $\pi^{n+1} g(X,Y)$ is invariant by $\sigma$ as well, and so we can write 
\[
\pi^{n+1} g(X,Y) \equiv  b+ f(X,Y) + \sigma\big(f(X,Y)\big) \mod \pi^{n+2}
\]
for some function $f$ in $\pi^{n+1} \OOO_V(V)$ and some element $b$ of $\pi^{n+1} R$ such that $\sigma(b)=b$.
Now set $\eta_{n+1}=\eta_n - f \; \text{ and } \; v_{n+1} = v_n + b.$
Then, modulo $\pi^{n+2}$ we have 
\begin{align*}
\eta_{n+1}\sigma(\eta_{n+1}) &\equiv (\eta_n- f)(\sigma(\eta_n)- \sigma(f))\\
 & \equiv \eta_n \sigma(\eta_n)-f \sigma(\eta_n)-\sigma(f)\eta_n \\
 & \equiv v_n(1+\pi U) + \pi^{n+1} g - f -\sigma(f)\\
 & \equiv v_n(1+\pi U) + b\\
 & \equiv v_{n+1}(1+\pi U). 
\end{align*}
The limits $\eta= \lim \eta_n$ and $v = \lim v_n$ exist in ${\OOO_V(V)}^\times$ and $R^\times$ respectively, and satisfy $\sigma(v)=v$ and $\eta \sigma(\eta) = v(1+\pi U)$, as desired.
\end{proof}



\section{Fractional annuli}
\label{section_fractional_annuli}

A closed $K$-analytic annulus of modulus $e$ is isomorphic to the subspace of $\A^{1, \mathrm{an}}_K$ defined by the inequalities $|\pi|^e \leq |X| \leq 1$, where $X$ is a coordinate of the analytic affine line. 
In this section, we consider a more general class of $K$-analytic spaces, that of annuli of fractional moduli, subspaces of $\A^{1, \mathrm{an}}_K$ defined by inequalities such as $|\pi|^\beta \leq |X| \leq |\pi|^\alpha$ with $\alpha, \beta \in \Q$.
\medskip

A $K$-analytic space $V$ is said to be a \emph{fractional open annulus} (or \emph{fractional semi-open annulus}, or \emph{fractional closed annulus}) if it is a semi-affinoid space having a model of the form $\Spf A$, with $A \cong R[[|\pi|^{-\alpha} X, |\pi|^{\beta} Y]]/(XY-1)$ (resp. $A \cong R\{|\pi|^{-\alpha} X\}[[|\pi|^{\beta} Y]]/(XY-1)$, or $A \cong R\{|\pi|^{-\alpha} X, |\pi|^{\beta} Y\}/(XY-1)$) for some rational numbers $\alpha, \beta$ such that $\alpha < \beta$.
  
When there is no need to specify the type (open, semi-open, or closed) of such a fractional annulus, we simply denote it by $V_{\alpha,\beta}$, and call it a fractional annulus of radii $|\pi|^\beta<|\pi|^\alpha$.
The \emph{modulus} of a fractional annulus $V_{\alpha,\beta}$ is the positive rational number $\beta-\alpha$.

\begin{remark}\label{remark_fractional_annuli_intrinsic_modulus}
Let $V_{\alpha,\beta}$ be a fractional $K$-annulus of radii $|\pi|^\beta<|\pi|^\alpha$, and let $\rho$ be a positive integer such that $\rho\alpha$ and $\rho\beta$ are both integers.
Then the base change of $V$ to the totally ramified degree $\rho$ extension $K(\pi^{\sfrac{1}{\rho}})$ of $K$ is a $K(\pi^{\sfrac{1}{\rho}})$-analytic annulus of the same type as $V_{\alpha,\beta}$ and of modulus $\rho(\beta-\alpha)$.
In particular, it follows from Remark~\ref{remark_modulus_intrinsic} that two isomorphic fractional annuli have the same modulus.
\end{remark}

\begin{remark}\label{remark_fractional_annuli_arent_annuli}
Annuli are clearly fractional annuli, but not every fractional annulus whose modulus is an integer is an annulus. 
For example, the open fractional annulus of radii $|\pi|^{\sfrac{1}{2}} < |\pi|^{-\sfrac{1}{2}}$ has modulus 1, but since it contains $K$-rational points it cannot be isomorphic to an open annulus of modulus 1.
In particular, this shows that two fractional annuli that have the same modulus are not necessarily isomorphic. 
\end{remark}

\begin{remark}
Let $a, b, a', b'$ be integers such that $(a,b)=(a', b')=1$ and $a/b < a'/b'$.
Observe that the special $R$-algebra that we used to construct the open fractional annulus of radii $|\pi|^{\sfrac{a'}{b'}} < |\pi|^{\sfrac{a}{b}}$, as defined in equation \eqref{notation_fractional_special_algebras}, is
\[
\resizebox{1\textwidth}{!}{$
\frac{R\big[\big[|\pi|^{-\alpha} X, |\pi|^{\beta} Y \big]\big]}{(XY-1)} = \frac{R[[U,V,W,Z]]}{\big (\pi^{b\lceil a/b \rceil -a} V - U^b, \pi^{{b'} \lceil {a'/b'} \rceil -{a'}} Z- W^{b'}, UW - \pi^{-\lceil a/b \rceil-\lceil a'/b' \rceil} \big )}.
$}
\]
The reader should be aware that this algebra is not flat in general, so the canonical reduction of the fractional annulus is not simply obtained by reduction modulo $\pi$ of the above equations.
For example, the open fractional annulus defined by $|\pi|^{\sfrac{1}{b}} < |X| < 1$ has model $\Spf A$, with
\[\
A= \frac{R[[W,Z,Y]]}{\big (\pi^{b-1} Z- W^b, WY - \pi \big )}.
\]
In $A$ we have $\pi^{b-1} (W-ZY^{b-1}) = 0$, and the canonical model of this fractional annulus is the formal spectrum of
\[
\frac{R[[W,Z,Y]]}{\big (\pi^{b-1} Z- W^b, WY - \pi, W-ZY^{b-1} \big )}= \frac{R[[Z,Y]]}{\big (\pi - ZY^b \big )},
\] 
hence its canonical reduction is the formal spectrum of ${k[[Z,Y]]}/{\big ( ZY^b \big )}$.
\end{remark}

We can be more precise than in Remarks~\ref{remark_fractional_annuli_intrinsic_modulus} and \ref{remark_fractional_annuli_arent_annuli}, and describe the moduli space of fractional annuli over $K$.
This is the content of the next proposition.
Consider the map
\begin{align*}
\Phi \colon \mbox{\{Fractional annuli over $K$ \}} & \longrightarrow  \Q_{>0} \times  \bigslant{\Q}{\Z}\\
 V_{\alpha, \beta} & \longmapsto (\beta-\alpha, \overline{\alpha})
\end{align*}
that sends a fractional annulus $V_{\alpha, \beta}$ of radii $|\pi|^\beta<|\pi|^\alpha$ to the pair consisting of its modulus and the class $\overline{\alpha}$ of $\alpha$ modulo $\Z$.

\begin{theorem}\label{theorem_moduli_space_fractional_annuli}
Let $P\in\{\mbox{open, semi-open, closed}\}$ denote a type of fractional annuli.
Then the map $\Phi$ induces a bijection
\[ 
\Phi \colon \bigg\{\,\parbox{14em}{Isomorphism classes of fractional \\ annuli of type $P$ over $K$} \, \bigg\} \stackrel{\sim}{\longrightarrow} \quot{\Big(\Q_{>0} \times \bigslant{\Q}{\Z}\Big)}{\sim_P}
\]
where $\sim_P$ is the equivalence relation on $\Q_{>0} \times \bigslant{\Q}{\Z}$ generated by the relations of the form $(a,b) \sim_P (a, a-b)$ if $P\in\{\mbox{open, closed}\}$, while $\sim_{\mbox{\tiny{semi-open}}}$ is the trivial equivalence relation.
\end{theorem}
\begin{proof}
We have already observed in Remark~\ref{remark_fractional_annuli_intrinsic_modulus} that isomorphic fractional annuli have the same modulus.
Moreover, it is clear that if two fractional annuli are isomorphic then they are of the same type $P$.
Therefore, to conclude the proof it remains to establish the following claim: two fractional annuli $V_{\alpha_1, \gamma-\alpha_1}$ and $V_{\alpha_2, \gamma-\alpha_2}$ of the same type $P$ and same modulus $\gamma$ are isomorphic if and only if $(\gamma,\overline{\alpha_1})\sim_P(\gamma,\overline{\alpha_2})$.
If $\overline{\alpha_1}=\overline{\alpha_2}$, then the multiplication of the coordinate $X$ by $\pi^{\alpha_2-\alpha_1}$ gives a $K$-isomorphism $V_{\alpha_1, \gamma-\alpha_1}\stackrel{\sim}{\to}V_{\alpha_2, \gamma-\alpha_2}$ as subspaces of $\A^{1,\mathrm{an}}_K$.
Similarly, if the annuli are either open or closed and $\overline{\alpha_1}=\overline{\gamma-\alpha_2}$, a $K$-isomorphism $V_{\alpha_1, \gamma-\alpha_1}\stackrel{\sim}{\to}V_{\alpha_2, \gamma-\alpha_2}$  as subspaces of $\A^{1,\mathrm{an}}_K\setminus\{0\}$ is obtained by first sending the coordinate $X$ to its inverse $X^{-1}$, then applying a transformation as in the previous case.
This proves the ``only if'' part of the claim above.
To prove the ``if'' part of the claim, let $\varphi\colon V_{\alpha_1, \gamma-\alpha_1}\stackrel{\sim}{\to}V_{\alpha_2, \gamma-\alpha_2}$ be an isomorphism of fractional annuli over $K$.
Since $\varphi$ induces an isomorphism at the level of $K$-rational points, it also induces an isomorphism between the biggest annuli contained in the two fractional annuli:
\[
\resizebox{1\textwidth}{!}{$
	V_1=\Big\{x\in V_{\alpha_1, \gamma-\alpha_1} \Big| |\pi|^{\lceil\gamma-\alpha_1\rceil} \leq |X| \leq |\pi|^{\lfloor\alpha_1\rfloor} \Big\}
	\stackrel{\sim}{\longrightarrow}
	V_2=\Big\{x\in V_{\alpha_2, \gamma-\alpha_2} \Big| |\pi|^{\lceil\gamma-\alpha_2\rceil} \leq |X| \leq |\pi|^{\lfloor\alpha_2\rfloor} \Big\},
$}
\]
and therefore $\varphi$ is an isomorphism between the complements $V_{\alpha_1, \gamma-\alpha_1}\setminus V_1$ and $V_{\alpha_2, \gamma-\alpha_2}\setminus V_2$.
Assume that the annuli are either open or closed and that $\alpha$ is not an integer; the remaining cases can be treated similarly and are left to the reader.
Observe that $V_{\alpha_1, \gamma-\alpha_1}\setminus V_1$ has a connected component $C$ that is a fractional annulus of modulus $\alpha_1 - \lfloor\alpha_1\rfloor$.
On the other hand, the connected components of $V_{\alpha_1, \gamma-\alpha_1}\setminus V_1$ are two fractional annuli of moduli $\alpha_2 - \lfloor\alpha_2\rfloor$ and $\gamma-\alpha_2 - \lceil\gamma-\alpha_2\rceil$, so $C$ has to be isomorphic to either of the two.
This implies that $\alpha_1 - \lfloor\alpha_1\rfloor$ equals either $\alpha_2 - \lfloor\alpha_2\rfloor$ or $\gamma-\alpha_2 - \lceil\gamma-\alpha_2\rceil$, which means precisely that $(\gamma,\overline{\alpha_1})\sim_P(\gamma,\overline{\alpha_2})$.
\end{proof}

\begin{remark}
As can be seen from the proof above, from the point of view of Berkovich spaces if $V_{\alpha,\beta}$ is a fractional annulus then $\overline{\alpha}$ measures the length of the segment of the skeleton of $V_{\alpha, \beta}$ that connects the boundary point corresponding to the component $(X)$ of the canonical reduction of $V_{\alpha,\beta}$ to the closest point of the skeleton onto which a $K$-rational point retracts.
\end{remark}


\section{Tame forms of annuli}
\label{section_forms_annuli}

We say that a $K$-analytic space $V$ is a \emph{$K$-form of an annulus} if there exists a finite extension $K'|K$ such that $V'=V\otimes_KK'$ is an annulus.
We then say that the extension $K'|K$ \emph{trivializes} the form $V$.
As we observed in Remark~\ref{remark_fractional_annuli_intrinsic_modulus}, fractional annuli over $K$ are $K$-forms of annuli; we say that a $K$-form of an annulus $V$ is itself \emph{trivial} if it is isomorphic to a fractional annulus over $K$.

In this section we apply Theorem~\ref{theorem_descent_model} and the results of Section~\ref{section_annuli} to classify forms of annuli for a large class of Galois extensions that includes totally tamely ramified extensions.
In particular, in Theorem~\ref{theorem_annuli} we show that all such forms are trivial as long as the action induced by the Galois group of $K'|K$ fixes the branches of the $K'$-annulus $V'$.
On the other hand, in Theorem~\ref{theorem_forms_switched_branches} we show that up to isomorphism there exists only one non-trivial form of an annulus of given modulus (except in the semi-open case, where all forms are trivial) trivialized by a quadratic extension that switches the branches of $V'$.

Observe that the corresponding problem for discs is well understood, since all tame forms of open polydiscs are trivial by \cite{Ducros13}, and all tame forms of closed discs are trivial by \cite{Schmidt15}. 
Our methods also allow to retrieve easily those results in dimension 1 for the extensions that we consider.
On the other hand, the triviality of forms of annuli when the Galois group fixes the branches has also recently been obtained by Chapuis, with different techniques, in \cite{Chapuis2017}. 
\medskip

We begin by studying the case when the action induced by the Galois group on the trivialized annulus $V\otimes_KK'$ fixes its branches.

\begin{theorem}\label{theorem_annuli}
Let $K'$ be a Galois extension of $K$ of degree $m$ and ramification index $\rho$ such that $K$ contains all $\rho$-th roots of unity, its residue characteristic does not divide $m$, and the residual extension $k'|k$ is solvable.
Let $V$ be a $K$-form of an annulus trivialized by $K'|K$.
Then, if every element of the Galois group of $K'|K$ fixes the branches of $V\otimes_KK'$, the form $V$ is trivial.
\end{theorem}

\begin{example}
If $K'$ is a Galois and totally tamely ramified extension of $K$ of degree $m$, then it satisfies the conditions of Theorem~\ref{theorem_annuli}, since it is generated by an $m$-th root $\varpi$ of a uniformizer $\pi$ of $K$, and a generator of the Galois group of $K'|K$ has to act on $\varpi$ by multiplication by a primitive $m$-th root of unity $\zeta\in K'$, so that a primitive $m$-th root of unity exists in $k'=k$.
\end{example}

\begin{proof}[Proof of Theorem~\ref{theorem_annuli}]
We only treat the case when $V$ is the form of an open annulus of modulus $e$.
The cases of semi-open and closed annuli are completely analogous, only requiring to replace some power series rings by the corresponding convergent version.
Let $\varpi$ be a uniformizer of $R'$ which is a root of $\pi$, and denote by $G$ the Galois group of $K'|K$.
By Theorem~\ref{theorem_descent_model}, if $(X,Y)$ is a Laurent pair for $V'$ then $V$ is isomorphic to the $K$-analytic space associated with the special formal $R$-scheme
\[
\X=\bigg(\prod\nolimits^\mathrm{dil}_{R'|R}\Spf\Big(\frac{R'[[X,Y]]}{(XY-\varpi^e)}\Big)\bigg)^G,
\]
where $e$ is the modulus of the open annulus $V\otimes_KK'$.
We now compute explicitly the special formal $R$-scheme above.
To simplify notations, let $A$ be the ring of the dilated Weil restriction of $\Spf\big(R'[[ X,Y]]/(XY-\varpi^e)\big)$ to $R$, so that $\X\cong\Spf(A_G)$.
We divide the rest of the proof into three steps.

\emph{Step (i)}.
Assume that the extension $K'|K$ is totally ramified.  
Then $R'$ is a free $R$-module of rank $m$, with basis $\{1,\varpi,\ldots,\varpi^{m-1}\}$, and the ring $A$ is obtained as follows. 
Write
\begin{align*}
X &= X_0+X_1\varpi+\ldots+X_{m-1}\varpi^{m-1},\\
Y &= Y_0+Y_1\varpi+\ldots+Y_{m-1}\varpi^{m-1}.
\end{align*}
Then, by Proposition~\ref{proposition_dilated_restriction_explicit} we have an isomorphism
\[
A\cong \frac{R\big[\big[X_0,|\varpi| X_{1}, |\varpi|^2 X_{2}, \ldots, |\varpi|^{m-1}  X_{m-1}, Y_0, |\varpi|  Y_{1},\ldots,|\varpi|^{m-1} Y_{m-1}\big]\big]}{J},
\]
where $J=(f_0, \dots, f_{m-1})$ is the ideal generated by the coefficients of the expansion of $XY-\varpi^e$ in the basis $\{1,\varpi,\ldots,\varpi^{m-1}\}$. 
Writing $e=a+bm$ for some $0\leq a<m$ and $b\geq0$, a simple computation of $\big(X_0+\ldots+X_{m-1}\varpi^{m-1}\big)\big(Y_0+\ldots+Y_{m-1}\varpi^{m-1}\big)-\varpi^e$ yields
\[
f_i=\sum_{j=0}^i X_j Y_{i-j} + \bigg( \sum_{j=i+1}^{m-1} X_j Y_{m+i-j} \bigg)\pi - \delta_{a,i}\pi^b
\]
for every $i=0,\ldots,m-1$, where $\delta_{a,i}$ is equal to $1$ if $a=i$ and to $0$ otherwise.
Denote by $\sigma$ a generator of $G$ and let $\zeta\in R$ be the primitive $m$-th root of unity such that $\sigma(\varpi)=\zeta\varpi$.
By Proposition~\ref{proposition_linearization} we can assume that $\sigma(X)=\zeta^\alpha X$ and $\sigma(Y)=\zeta^\beta Y$ for some $0\leq \alpha,\beta \leq m-1$.
Observe that then $\alpha+\beta$ must be congruent to $e$ modulo $m$. 
It follows that, by a similar computation as the one performed in the proof of part $(i)$ of Lemma~\ref{lemma_computation_G-fix_Weil_basechange}, $\sigma(X_i)=\zeta^{\alpha-i} X_i$ and $\sigma(Y_i)=\zeta^{\beta-i} Y_i$ for every $i=0,\ldots,m-1$. 
The kernel of the surjection $A\to A_G$ is generated by the monomials $\big(1-\zeta^{\alpha-i}\big)X_i$ and $\big(1-\zeta^{\beta-i}\big)Y_i$, therefore only $X_\alpha$ and $Y_\beta$ survive in $A_G$.
It follows that $f_i=0$ for every $i\neq a$, and
\[
f_a=
\begin{cases}
X_\alpha Y_\beta-\pi^b & \text{ if }\alpha+\beta=a,\\
X_\alpha Y_\beta\pi-\pi^{b} & \text{ if }\alpha+\beta=m+a.
\end{cases}
\]
In both cases we see that $V$ is a fractional annulus, which is what we wanted to show.

\emph{Step (ii)}.
Assume that $K'|K$ is cyclic and unramified.
As in the proof of part $(ii)$ of Lemma~\ref{lemma_computation_G-fix_Weil_basechange}, we can apply the normal basis theorem (\cite[VI, \S13, Theorem 13.1]{Lang}) to find $\alpha\in R'$ such that $|\alpha|=1$ and the set $\big(\alpha, \sigma(\alpha),\ldots, \sigma^{m-1}(\alpha)\big)$ is a basis of $R'$ as a free $R$-module, where $\sigma$ is a generator of $G$.
Moreover, by dividing $\alpha$ by its trace $\sum_{i=0}^{m-1}\sigma^i(\alpha)\neq0$ we can assume that the trace of $\alpha$ is $1$, so that an element $a$ of $R$, when seen as an element of $R'$, can be expressed as $a = \sum_{i=0}^{m-1}a\sigma^i(\alpha)$ in the basis above.
Let $T$ be one of the variables among $X$ and $Y$, and write
\[
T=T_0\alpha+T_1\sigma(\alpha)+\ldots+T_{m-1}\sigma^{m-1}(\alpha),
\]
so that as before we have an isomorphism
\[
A\cong {R[[X_0, X_1, \ldots, X_{m-1}, Y_0, \ldots, Y_{m-1}]]}/{J},
\]
where $J$ is the ideal of the coefficients of $XY-\varpi^e$ expressed in the basis $\big(\alpha, \sigma(\alpha),\ldots, \sigma^{m-1}(\alpha)\big)$.
By Proposition~\ref{proposition_linearization} we can assume that 
$\sigma(T) = T$.
Since we also have $\sigma(T) = \sigma(T_0)\sigma(\alpha) + \sigma(T_1)\sigma^2(\alpha)+\ldots + \sigma(T_{m-1})\alpha$,
in the ring $A_G$ we have $X_0=X_1=\cdots=X_{m-1}$ and $Y_0=Y_1=\cdots=Y_{m-1}$, and therefore 
\[
XY-\varpi^e=\Big(X_0\sum\nolimits_i\sigma^i(\alpha)\Big)\Big(Y_0\sum\nolimits_i\sigma^i(\alpha)\Big)-\pi^e\sum\nolimits_i\sigma^i(\alpha),\]
 so that all the generators of $J$ become $X_0Y_0-\pi^e$.
This proves that $A_G\cong R[[X_0,Y_0]]/(X_0Y_0-\pi^e)$, that is $V$ is an open annulus of modulus $e$.

\emph{Step (iii)}.
Let us treat the general case.
Up to adding to $K'$ a $\rho$-th root $\pi^{\sfrac{1}{\rho}}$ of $\pi$, we can assume that the extension $K'|K$ decomposes as an extension $K'|K(\pi^{\sfrac{1}{\rho}})$ of the totally ramified extension $K(\pi^{\sfrac{1}{\rho}})|K$.
The extension $K'|K(\pi^{\sfrac{1}{\rho}})$ is Galois, unramified and solvable, and so it can be decomposed as a tower of unramified, cyclic Galois extensions.
By applying the result we proved in step $(ii)$ to each extension in this tower, we deduce that the $K(\pi^{\sfrac{1}{\rho}})$-form $V\otimes_KK(\pi^{\sfrac{1}{\rho}})$ of the $K'$-analytic annulus $V\otimes_KK'$ is itself an annulus.
Finally, the extension $K(\pi^{\sfrac{1}{\rho}})|K$ is totally ramified, cyclic Galois, because $K$ contains all $\rho$-th roots of unity, so that the result we obtained in step $(i)$ shows that the $K$-form $V$ of the annulus $V\otimes_KK(\pi^{\sfrac{1}{\rho}})$ is a fractional annulus, concluding the proof of the theorem.
\end{proof}

If $K'$ is an extension of $K$ satisfying the hypotheses of the theorem above, this result, together with Theorem~\ref{theorem_moduli_space_fractional_annuli}, classifies completely the $K$-forms of a $K'$-annulus $X$ of modulus $e$ with fixed branches.
Those are all the fractional annuli of radii $|\pi|^\beta<|\pi|^\alpha$ with $\beta-\alpha=e/\rho$ and $\alpha,\beta\in \frac{1}{\rho}\Z$, where $\rho$ is the ramification index of $K'|K$. 
In the language of group cohomology, the set of isomorphism classes of these forms is $\mathrm{H}^1\big(\Gal(K'|K),\mathrm{Aut}^\mathrm{fix}_{K'}(X)\big)$, where $\mathrm{Aut}^\mathrm{fix}_{K'}(X)$ is the normal subgroup of $\mathrm{Aut}_{K'}(X)$ that consists of the $K'$-automorphisms of $X$ fixing its branches.

If $V$ is a $K$-form of an annulus trivialized by an extension $K'|K$, and if there exists an element of the Galois group of $K'$ over $K$ that switches the branches of $V\otimes_KK'$, then the $K$-form $V$ cannot be trivial.
The next proposition shows that such non-trivial forms (both ramified and unramified) indeed exist, and classifies those that are trivialized by a quadratic extension.

\begin{theorem}\label{theorem_forms_switched_branches}
Let $K'=K\big(\sqrt a\big)$ be a quadratic extension of $K$ of ramification index $\rho\in\{1,2\}$, and assume that the residue characteristic of $K$ is different from $2$.
Let $V'$ be a $K'$-analytic annulus of modulus $e$.
Then there exists a $K$-form $V$ satisfying $V\otimes_KK'\cong V'$ and such that the action induced by the generator of the Galois group of $K'|K$ switches the branches of $V'$ if and only if $V'$ is either an open or a closed annulus and $\rho$ divides $e$.
Moreover, when this is the case we have
\[
V \cong \Spf\bigg(\frac{R\big[\big[ X, |a|^{\sfrac{1}{2}} Y \big]\big]}{{(X^2- a Y^2 + u\pi^{\sfrac{e}{\rho}})}}\bigg)^\beth
\]
whenever $V'$ is open, and
\[
V \cong \Spf\bigg(\frac{R\big\{X,|a|^{\sfrac{1}{2}} Y\big\}}{{(X^2- a Y^2 + u\pi^{\sfrac{e}{\rho}})}}\bigg)^\beth
\]
whenever $V'$ is closed, for some unit $u$ of $R$.
\end{theorem}

\begin{proof}
Let $V$ be a $K$-form such that $V\otimes_KK'\cong V'$.
Proposition~\ref{proposition_linearization} implies that $V'$ is either an open or a closed annulus.
We focus on the former case, as the other one is completely analogous.
Let $\sigma$ be a generator of the Galois group of $K'$ over $K$.
We treat separately two cases, according to whether the extension $K'|K$ is ramified or unramified, and proceed as in the proof of Theorem~\ref{theorem_annuli}, computing explicitly the model $\Spf A_{\left<\sigma\right>}$ of $V$, where $\Spf A= \prod\nolimits^\mathrm{dil}_{R'|R}\Spf\big(\calO^\circ_V(V)\big)$.
Given the similarities with that proof, some details are left to the reader.

\emph{Step (i).} 
If $K'$ is a ramified extension of $K$, we have $R'\cong R\oplus R\varpi$, with $\varpi^2=\pi$, and $\sigma$ acts on $R'$ by multiplying $\varpi$ by $-1$.
Again by Proposition~\ref{proposition_linearization}, the modulus $e$ of $V\otimes_KK'$ is even, and we can write 
$\calO^\circ_V(V) \cong R'[[X,Y]]/(XY-u\varpi^e)$, with $\sigma(X)= Y$ and $\sigma(Y)=X$, and where $u$ is a unit of $R$.
We write $X=X_0+X_1\varpi$ and $Y=Y_0+Y_1\varpi$, so that we have $A \cong R[[X_0,|\varpi| X_{1}, Y_0, |\varpi|  Y_{1}]]/(X_1Y_1+X_1Y_0, X_0Y_0-\pi X_1Y_1-u\pi^{\sfrac{e}{2}})$.
Since we have $\sigma(X_0)=Y_0$ and $\sigma(X_1)=-Y_1$, we deduce that $A_{\left<\sigma\right>} \cong R[[X_0,|\varpi| X_{1}]]/(X_0^2-\pi X_1^2-u\pi^{\sfrac{e}{2}})$, which is what we wanted.

\emph{Step (ii).} If $K'$ is an unramified extension of $K$, we can write $R'=R\oplus R\alpha$ for some $\alpha$ in $R'$ such that $\alpha^2\in R$ and $\sigma(\alpha)=-\alpha$.
We can now apply one last time Proposition~\ref{proposition_linearization}, and the same computations as in Step~(ii) yield as expected $A_{\left<\sigma\right>} \cong R[[X_0,X_{1}]]/(X_0^2-\alpha^2 X_1^2-u\pi^{e})$.

\emph{Step (iii).}
It remains to show that $(\Spf B)^\beth$, for 
$$B={R[[ X, |a|^{\sfrac{1}{2}} Y ]]}/{{(X^2- a Y^2 + u\pi^{\sfrac{e}{\rho}})}},$$
is indeed the $K$-form we expect.
Since $R'=R[a^{\sfrac{1}{2}}]$, with $a\in R$, we have
\[
B\otimes_RR' 
\cong 
\frac{R'[[X,|a|^{\sfrac{1}{2}} Y]]}{(X^2-a Y^2-u\pi^{\sfrac{e}{\rho}})} 
\cong 
\frac{R'[[X,|a|^{\sfrac{1}{2}} Y]]}{\big((X+a^{\sfrac{1}{2}} Y)(X-a^{\sfrac{1}{2}} Y)-u\pi^{\sfrac{e}{\rho}} \big)}.
\]
This is the algebra of bounded functions on an open annulus of modulus $e$, as the change of variables $X'=X+a^{\sfrac{1}{2}} Y$, $Y'=u^{-1}(X-a^{\sfrac{1}{2}} Y)$ shows. 
The automorphism $\sigma$ keeps $X$ fixed but changes the sign of $Y$, and therefore it exchanges the ideals $(X')$ and $(Y')$, which is the last thing that we needed to verify.
\end{proof}

Together with Theorem~\ref{theorem_annuli}, this result implies that, if $K'|K$ is a quadratic extension and the residue characteristic is different from 2, then the cardinality of the set $\mathrm{H}^1\big(\Gal(K'|K),\mathrm{Aut}_{K'}(X)\big)$ of $K$-forms of a $K'$-annulus $X$ verifies:
\[
\left|\mathrm{H}^1\big(\Gal(K'|K),\mathrm{Aut}_{K'}(X)\big)\right|=
\begin{cases}
3 & \text{if $X$ is not semi-open, $\rho=2$,}\\[-4pt] 
  & \text{and $e$ is even,} \\[3pt]
2 & \text{if $X$ is semi-open and $\rho=2$,}\\[-4pt]
  & \text{or $X$ is not semi-open and $\rho=1$,} \\[3pt]
1 & \text{otherwise.}
\end{cases}
\]

\begin{remark}\label{remark_nontrivial_forms}
Let $V$ be a non-trivial $K$-form of an annulus as in Theorem~\ref{theorem_forms_switched_branches}.
Since $V$ is homeomorphic to $V'/\Gal(K'|K)$, and since the Galois action flips the skeleton of the $K'$-annulus $V'$, as observed in Remark~\ref{remark_switched_skeleton}, it follows that $V$ has exactly one boundary point.
An algebraic way to see this consists of showing that the canonical reduction of $V$ is irreducible.
For example, if $V'$ is a closed annulus of modulus $e$ and $K'$ is unramified over $K$, then the formal spectrum of $R\{X, Y\}/(X^2- a Y^2 + u\pi^e)$ is the canonical model of $V$, and so its canonical reduction is the spectrum of $k[X,Y]/(X^2- a Y^2)$, that is irreducible.
\end{remark}

\begin{remark}\label{remark_polydiscs}
The methods of this section allow to study forms of open and closed polydiscs as well.
Indeed, as soon as we know that a Galois action can be made linear on the coordinates of a polydisc, in the same spirit as our Proposition~\ref{proposition_linearization} does for annuli, then the arguments of Theorem~\ref{theorem_annuli} permit to prove the triviality of forms.
While in dimension $1$ this is always satisfied (and it can be shown similarly as in Proposition~\ref{proposition_linearization}), in general whether any tame polynomial action of finite order on the affine $n$-space is linearizable is a well known open problem (see \cite[\S5,6]{Kraft1996} for a thorough discussion).
In particular, beyond the dimension 2 it is currently not known whether non-trivial tame forms of closed polydiscs exist.
Observe that the linearizability assumption is equivalent to the hypothesis of \emph{residually affine} action that Chapuis assumes in \cite[Théorème 2.12]{Chapuis2017} to show the triviality of forms of polydiscs, since if the action is residually affine one can find a suitable lift that can be linearized after a change of coordinates of the polydisc.
\end{remark}

The results of this section can be applied to obtain a non-archimedean analytic proof of a seemingly unrelated result: the existence of resolutions of singularities of surfaces over an algebraically closed field of characteristic zero.
This was first proven by Zariski \cite{Zariski1939}.
We briefly explain how this can be done, building on \cite{Fantini2017}.

\begin{theorem}\label{theorem_resolutions}
Let $k$ be an algebraically closed field of characteristic zero and let $X$ be a surface over $k$.
Then $X$ admits a resolution of its singularities.
\end{theorem}

\begin{proof}
Without loss of generality, by replacing $X$ by its normalization and reasoning locally, we can assume that $X$ is normal and $x\in X$ is its only singular point.
With this data, in \cite{Fantini2017} is associated a locally ringed space $\NL(X,x)$, the non-archimedean link of $x$ in $X$.
Using \cite[Corollary 4.10]{Fantini2017}, we can assume without loss of generality that $\NL(X,x)$ is a smooth Berkovich curve over a discretely valued field of the form $k((\pi))$.
By \cite[Proposition 10.9]{Fantini2017}, a resolution of $(X,x)$ exists if and only if we can find a finite and non-empty set $S$ of type 2 points of $\NL(X,x)$ such that for each connected component $V$ of $\NL(X,x)\setminus S$ the ring $\calO^\circ(V)$ is regular.
By applying \cite[Théorème 5.1.14.(iv)]{Ducros} to $\NL(X,x)$, we deduce that there exists a finite set $S'$ of type $2$ points of $\NL(X,x)$ such that each connected component $V$ of $\NL(X,x)\setminus S'$ satisfies the following condition: there exists a finite separable extension $\mathfrak{s}(V)$ of $k((\pi))$ which consists of analytic functions on $V$, so that $V$ can be seen as an analytic space over $\mathfrak{s}(V)$, and $V$ is either a $\mathfrak{s}(V)$-form of an open disc or a $\mathfrak{s}(V)$-form of an open annulus whose trivializing extension fixes the branches (in \emph{loc. cit.}, see $3.1.1.4$ for the definition of $\mathfrak{s}(X)$).
Since the non-archimedean link $\NL(X,x)$ does not depend on the choice of a base field, and $\mathfrak{s}(V)$ is abstractly isomorphic to $k((\pi))$ itself, without loss of generality we can assume that $V$ is a $k((\pi))$-form of an open disc or of an open annulus with fixed branches.
Assume that $V$ is a form of an annulus.
Since $k$ is algebraically closed and of characteristic zero, the extension of $k((\pi))$ trivializing $V$ is of the form $K'=k((\varpi))$, where  $\varpi$ is a root of $\pi$, and satisfies the hypotheses of Theorem~\ref{theorem_annuli}, so that $V$ is a fractional annulus over $k((\pi))$.
By adding more type 2 points to $S'$ we can cut $V$ into smaller fractional annuli and therefore assume that $V'$ is a $K'$-annulus of modulus one, so that in particular $\calO^\circ(V') \cong k[[\varpi]][[X,Y]]/(XY-\varpi)\cong k[[X,Y]]$ is regular.
Then \cite[\S0, Proposition 17.3.3.(i)]{EGA4.1}, applied to the flat local morphism of local noetherian rings $\calO^\circ(V) \to \calO(V)\otimes_{k[[\pi]]}k[[\varpi]] \cong \calO^\circ(V')$, ensures that $\calO^\circ(V)$ is regular.
The same result shows that the algebras of bounded functions on $K$-forms of open discs are regular as well, which concludes the proof.
\end{proof}


\bibliographystyle{alpha}                              
\bibliography{biblioarboreti}

\vfill

\end{document}